\documentclass[11pt]{article} 
\usepackage{amsfonts}
\usepackage{amsmath}
\usepackage{amssymb}
\usepackage{tikz}
\usepackage{mathabx}
\usepackage{float}
\usepackage{mathtools}
\usepackage[symbol*]{footmisc}

\usepackage{authblk}
\usepackage[margin=1in]{geometry}

\usepackage{amssymb}
\usepackage{amsthm}
\usepackage{dsfont}
\usepackage{graphicx}
\usepackage{enumerate}
\usepackage[hypcap]{caption}
\usepackage{color, soul}
\usepackage{float}
\usepackage[ruled]{algorithm2e}
\usepackage{prettyref}
\usepackage{subcaption}
\usepackage{comment}
\usepackage{needspace}

\newcommand{\PP}{\mathbb{P}}
\newcommand{\QQ}{\mathbb{Q}}
\newcommand{\NN}{\mathbb{N}}

\newcommand{\RR}{\mathbb{R}}
\newcommand{\E}{\mathbb{E}}
\newcommand{\EE}{\mathbb{E}}
\DeclareMathOperator*{\Var}{\mathrm{Var}}

\newcommand{\Adv}{\mathsf{Adv}}
\newcommand{\Corr}{\mathsf{Corr}}
\newcommand{\NCorr}{\mathsf{Corr'}}

\newcommand{\ER}{Erd\H{o}s--R\'{e}nyi }
\newcommand{\ERend}{Erd\H{o}s--R\'{e}nyi}
\newcommand{\sY}{R}

\newcommand{\tr}{\widetilde{r}}
\newcommand{\tX}{\widetilde{X}}
\newcommand{\eps}{\varepsilon}

\newcommand{\Ntri}{N^{\rm tri}}
\newcommand{\mono}{C_c}

\newtheorem{theorem}{Theorem}
\newtheorem{remark}[theorem]{Remark}
\newtheorem{lemma}[theorem]{Lemma}

\newtheorem{conjecture}[theorem]{Conjecture}

\newtheorem{definition}[theorem]{Definition}
\newtheorem{proposition}[theorem]{Proposition}

\numberwithin{theorem}{section}
\numberwithin{equation}{section}
\numberwithin{figure}{section}

\usepackage{fancybox}

\newenvironment{fminipage}%
  {\begin{Sbox}\begin{minipage}}%
  {\end{minipage}\end{Sbox}\fbox{\TheSbox}}

\newenvironment{algbox}[0]{\vskip 0.2in
\noindent 
\begin{fminipage}{6.3in}
}{
\end{fminipage}
\vskip 0.2in
}

\newcommand{\edge}{
  \begin{tikzpicture}[baseline=-0.6ex,scale=0.25]
  \tikzstyle{vertex}=[circle,fill=black, minimum size=2pt,inner sep=1pt]
  \node[vertex] (v1) at (0, -0.5){};
  \node[vertex] (v2) at (0,0.5){};
  \draw (v1)--(v2);
  \end{tikzpicture}}

\newcommand{\edgedotted}{
  \begin{tikzpicture}[baseline=-0.6ex,scale=0.25]
  \tikzstyle{vertex}=[circle,fill=black, minimum size=2pt,inner sep=1pt]
  \node[vertex] (v1) at (0, -0.5){};
  \node[vertex] (v2) at (0,0.5){};
  \draw[densely dotted, line width=0.4mm] (v1)--(v2);
  \end{tikzpicture}}

\newcommand{\trian}{
  \begin{tikzpicture}[baseline=-0.3ex,scale=0.25]
  \tikzstyle{vertex}=[circle,fill=black, minimum size=2pt,inner sep=1pt]
  \node[vertex] (v1) at (-0.5, 0){};
  \node[vertex] (v2) at (0.5,0){};
  \node[vertex] (v3) at (0,0.8){};
  \draw (v1)--(v2)--(v3)--(v1);
  \end{tikzpicture}}

\newcommand{\gbowtiecol}{
  \begin{tikzpicture}[baseline=-0.3ex,scale=0.25]
  \tikzstyle{vertex}=[circle,fill=black, minimum size=2pt,inner sep=1pt]
  \node[vertex] (vL1) at (-0.8, 1){};
  \node[vertex] (vL2) at (-0.8,0){};
  \node[vertex] (vc) at (0,0.5){};
  \node[vertex] (vR1) at (0.8,1){};
  \node[vertex] (vR2) at (0.8,0){};
  \draw[color=gray] (vc)--(vL1)--(vL2)--(vc);
  \draw[color=pink] (vc)--(vR1)--(vR2)--(vc);
  \end{tikzpicture}}

\newcommand{\hedge}{
  \begin{tikzpicture}[baseline=-0.3ex,scale=0.25]
  \tikzstyle{vertex}=[circle,fill=black, minimum size=2pt,inner sep=1pt]
  \node[vertex] (v1) at (-0.5, 0){};
  \node[vertex] (v2) at (0.5,0){};
  \draw (v1)--(v2);
  \end{tikzpicture}}

\newcommand{\dumbedge}{
  \begin{tikzpicture}[baseline=-0.3ex,scale=0.25]
  \tikzstyle{vertex}=[circle,fill=black, minimum size=2pt,inner sep=1pt]
  \node[vertex] (v1) at (-0.5, 0){};
  \node[vertex] (v2) at (0.5,0){};
  \node[vertex] (v3) at (0,0.8){};
  \draw (v1)--(v2);
  \end{tikzpicture}}

\newcommand{\dtrian}{
  \begin{tikzpicture}[baseline=-0.3ex,scale=0.25]
  \tikzstyle{vertex}=[circle,fill=black, minimum size=2pt,inner sep=1pt]
  \node[vertex] (v1) at (-0.5, 0){};
  \node[vertex] (v2) at (0.5,0){};
  \node[vertex] (v3) at (0,0.8){};
  \draw (v3) edge[bend left] (v1);
  \draw (v3) edge[bend right] (v1);
  \draw (v3)--(v2)--(v1);
  \end{tikzpicture}}

\newcommand{\dlopsidedlollipop}{
  \begin{tikzpicture}[baseline=-0.3ex,scale=0.25]
  \tikzstyle{vertex}=[circle,fill=black, minimum size=2pt,inner sep=1pt]
  \node[vertex] (v1) at (-0.5, 0.2){};
  \node[vertex] (v2) at (0.5,0.2){};
  \draw (v1) to [out=180,in=90,looseness=8] (v1);
  \draw (v2) -- (v1);
  \end{tikzpicture}}

\newcommand{\dddtriancol}{
  \begin{tikzpicture}[baseline=-0.3ex,scale=0.25]
  \tikzstyle{vertex}=[circle,fill=black, minimum size=2pt,inner sep=1pt]
  \node[vertex] (v1) at (-0.625, 0){};
  \node[vertex] (v2) at (0.625,0){}; 
  \node[vertex] (v3) at (0,1){};
  \draw (v1) edge[bend left, color=gray] (v2);
  \draw (v1) edge[bend right, color=pink] (v2);
  \draw (v1) edge[bend left, color=pink] (v3);
  \draw (v1) edge[bend right, color=gray] (v3);
  \draw (v2) edge[bend left, color=gray] (v3);
  \draw (v2) edge[bend right, color=pink] (v3);
  \end{tikzpicture}}

\newcommand{\dkitecol}{
  \begin{tikzpicture}[baseline=-0.3ex,scale=0.25]
  \tikzstyle{vertex}=[circle,fill=black, minimum size=2pt,inner sep=1pt]
  \node[vertex] (v1) at (0, 1){};
  \node[vertex] (v2) at (0,0){};
  \node[vertex] (vL) at (-0.8,0.5){};
  \node[vertex] (vR) at (0.8,0.5){};
  \draw[color=gray] (v2) -- (vL) -- (v1);
  \draw (v2) edge[bend left, color=gray] (v1);
  \draw (v2) edge[bend right, color=pink] (v1);
  \draw[color=pink] (v1) -- (vR) -- (v2);
  \end{tikzpicture}}

  \newcommand{\dcherrya}{
  \begin{tikzpicture}[baseline=-0.3ex,scale=0.25]
  \tikzstyle{vertex}=[circle,fill=black, minimum size=2pt,inner sep=1pt]
  \node[vertex] (v1) at (-0.5, 0){};
  \node[vertex] (v2) at (0.5,0){};
  \node[vertex] (v3) at (0,0.8){};
  \draw (v3) edge[bend left] (v1);
  \draw (v3)--(v2);
  \end{tikzpicture}}

  \newcommand{\dcherryb}{
  \begin{tikzpicture}[baseline=-0.3ex,scale=0.25]
  \tikzstyle{vertex}=[circle,fill=black, minimum size=2pt,inner sep=1pt]
  \node[vertex] (v1) at (-0.5, 0){};
  \node[vertex] (v2) at (0.5,0){};
  \node[vertex] (v3) at (0,0.8){};
  \draw (v3) edge[bend right] (v1);
  \draw (v3)--(v2);
  \end{tikzpicture}}

\newcommand{\cherry}{
  \begin{tikzpicture}[baseline=-0.3ex,scale=0.25]
  \tikzstyle{vertex}=[circle,fill=black, minimum size=2pt,inner sep=1pt]
  \node[vertex] (v1) at (-0.5, 0){};
  \node[vertex] (v2) at (0.5,0){};
  \node[vertex] (v3) at (0,0.8){};
  \draw (v1)--(v3)--(v2);
  \end{tikzpicture}}

\newcommand{\sidecherry}{
  \begin{tikzpicture}[baseline=-0.3ex,scale=0.25]
  \tikzstyle{vertex}=[circle,fill=black, minimum size=2pt,inner sep=1pt]
  \node[vertex] (v1) at (-0.5, 0){};
  \node[vertex] (v2) at (0.5,0){};
  \node[vertex] (v3) at (0,0.8){};
  \draw (v3)--(v1)--(v2);
  \end{tikzpicture}}

\title{ 
Is it easier to count communities than find them?\footnotemark
}

\begin{document}

\author[1]{Cynthia Rush\thanks{Email: \textit{cynthia.rush@dcolumbia.edu}. Part of this work was supported by NSF CCF $\#1849883$ and part of the work was done while visiting the Simons Institute for the Theory of Computing, supported by a Google Research Fellowship.}}

\author[2]{Fiona Skerman\thanks{Email: \textit{fiona.skerman@math.uu.se}. Partially supported by the Wallenberg AI, Autonomous Systems and Software Program WASP and the project AI4Research at Uppsala University. Part of this work was done while visiting the Simons Institute for the Theory of Computing, supported by a Simons-Berkeley Research Fellowship.}}

\author[3]{Alexander S.\ Wein\thanks{Email: \textit{aswein@ucdavis.edu}. Partially supported by a Sloan Research Fellowship and NSF CAREER Award CCF-2338091. Part of this work was done at Georgia Tech, supported by NSF grants CCF-2007443 and CCF-2106444. Part of this work was done while visiting the Simons Institute for the Theory of Computing, supported by a Simons-Berkeley Research Fellowship.}}

\author[4]{Dana Yang\thanks{Email: \textit{dana.yang@cornell.edu}. Part of this work was done while visiting the Simons Institute for the Theory of Computing, supported by a Simons-Berkeley Research Fellowship.}}
%
\affil[1]{Department of Statistics, Columbia University.}
\affil[2]{Department of Mathematics, Uppsala University.}
\affil[3]{Department of Mathematics, University of California, Davis.}
\affil[4]{Department of Statistics and Data Science, Cornell University.}

\maketitle

\footnote{This is an extended version of a conference paper in ITCS. See Section~\ref{sec:new} for the new contributions of this version.}

\begin{abstract}
Random graph models with community structure have been studied extensively in the literature. For both the problems of detecting and recovering community structure, an interesting landscape of statistical and computational phase transitions has emerged. A natural unanswered question is: might it be possible to infer properties of the community structure (for instance, the number and sizes of communities) even in situations where actually finding those communities is believed to be computationally hard? We show the answer is no. In particular, we consider certain hypothesis testing problems between models with different community structures, and we show (in the low-degree polynomial framework) that testing between two options is as hard as finding the communities.

Our methods give the first computational lower bounds for testing between two different ``planted'' distributions, whereas previous results have considered testing between a planted distribution and an i.i.d.\ ``null'' distribution. We also show a formal relationship between the low--degree frameworks for recovery in a planted model and for testing two planted models.\\
\vspace{10mm}
\end{abstract}


\needspace{3\baselineskip}
\section{Introduction}

Questions of detecting and recovering community structure in random graph models have been studied extensively in the literature. Popular models include the \emph{planted dense subgraph model}~\cite{ACV,HWX-comp}, where an \ER base graph is augmented by adding one or more ``communities'' --- subsets of vertices with a higher-than-average connection probability between them --- and the \emph{stochastic block model} (see~\cite{abbe-survey,moore-survey} for a survey). There are by now a multitude of results identifying sharp conditions based on the problem parameters, e.g.\ edge probabilities and number/sizes of communities, under which it is possible (or impossible) to recover (exactly or approximately) the hidden partition of vertices, given a realization of the graph as input. Notably, many settings are believed to exhibit a \emph{statistical-computational gap}; that is, there exists a ``possible but hard'' regime of parameters where it is \emph{statistically} possible to recover the communities (typically by brute-force search) but there is no known \emph{computationally efficient}, meaning polynomial-time, algorithm for doing so. It may be that this hardness is inherent, meaning no poly-time algorithm exists, which is suggested by a growing body of ``rigorous evidence'' including reductions from the \emph{planted clique} problem~\cite{BBH-reduction,HWX-comp} and limitations of known classes of algorithms~\cite{local-stats,sbm-phys,HS-bayesian,SW-estimation}.

Despite all this progress, one question that remains relatively unexplored is the following: in the aforementioned ``hard'' regime, even though it seems hard to recover the communities, might it still be possible to learn \emph{something} about the community structure (e.g., the number or sizes of communities)? After all, in some models it has already been established that \emph{detecting} the presence of a dense subgraph (i.e., distinguishing the planted subgraph model from an appropriate \ER ``null'' model) appears to be strictly easier than actually recovering which vertices belong to it~\cite{BBH-reduction,CX-tradeoffs,HWX-comp,SW-estimation}. Existing detection-recovery gaps of this nature often occur due to a ``trivial'' test for detection (e.g., the total edge count), and the motivation for our work is to understand more precisely which properties of the community structure can be inferred in the hard regime, and which ones cannot.

\paragraph{A simple testing problem.}
One of the simplest inference tasks on the community structure is to detect the number of communities. Let us consider a toy problem of testing between two graph models: under $\mathbb{P}$ the graph contains one community of expected size $k$, while under $\mathbb{Q}$ the graph contains two communities each of expected size $k/2$. The community membership of each vertex is independent in both models ($k/n$ under $\mathbb{P}$ and $k/(2n)$, $k/(2n)$ under $\mathbb{Q}$) and vertices cannot be members of more than one community. Suppose any pair of vertices from the same community are connected independently with probability $2q$ and $3q$ under $\mathbb{P}$ and $\mathbb{Q}$, respectively, and all the other pairs of vertices are connected independently with probability $q$ under both models. Such a parameterization matches the expected degrees of the nodes under the two distributions, so that a simple test based on the total edge count fails to distinguish between $\mathbb{P}$ and $\mathbb{Q}$. One natural test is to threshold the number of triangles. It is easy to derive that the expected number of triangles under $\mathbb{P}$ and $\mathbb{Q}$ scale as different constant multiples of $q^3 k^3$, and the variance of the number of triangles is of order $\Theta(n^3q^3)$ under both models. Thus, the simple triangle counting algorithm consistently distinguishes $\mathbb{P}$ and $\mathbb{Q}$ if $q^3k^3\gg \sqrt{n^3q^3}$, i.e.\ $qk^2/n\gg 1$.

It is intriguing that the condition for the triangle counting algorithm to succeed coincides with the conjectured computational barrier for the more difficult task of finding all members of the community under the model $\PP$~\cite{SW-estimation}. In other words, in the entire ``hard'' regime where one cannot efficiently locate the planted community, the triangle counting algorithm fails to even tell whether the graph contains one or two communities. In this paper, we show that this statement extends beyond the simple triangle counting algorithm to all low-degree tests. Our main result is given in the following (informal) theorem statement.

\begin{theorem}[Informal]
\label{thm:informal}
If $q(k^2/n\vee 1)\leq 1/\mathrm{polylog}(n)$, then no low-degree test consistently tests between the graph models with one and two planted communities.
\end{theorem}

Moreover, the informal result of Theorem~\ref{thm:informal} extends to a much wider class of testing problems than those for which it is stated. We find that, whenever recovery is computationally hard, all low-degree tests fail to distinguish models with different numbers of planted communities of possibly different sizes.
In other words, inferring the community structure is just as hard as finding members of the planted communities themselves. We show a similar phenomenon for graphs with Gaussian weights. See Theorems~\ref{thm:Gaussian} and~\ref{thm:binary} for the formal statements. It is important to note that our results apply even in regimes where it is easy to distinguish $\PP$ (or $\QQ$) from an \ER graph; that is, one cannot recover our results simply by arguing that both $\PP$ and $\QQ$ are hard to distinguish from \ERend.

We additionally give a few other related results.


\paragraph{Connections between detection and recovery in the low-degree framework.}

We show a connection between detection and recovery. For a given recovery problem in a planted model there is an equivalent testing problem where one tests between two planted models. In the other direction, for a testing problem between two planted models there is an equivalent recovery problem if the likelihood ratio of the signals exists.  This equivalence is in a strong sense: one is low-degree hard if and only if the other one is low-degree hard. We will see that 
testing between planted distributions generalizes recovery in the low-degree framework. See Section~\ref{sec:formal_equiv} for a full statement of these results.

\paragraph{Alternative proof strategy for hardness of recovery.}

We give a reduction showing that if there were an algorithm that successfully \emph{recovers} a planted community, one could turn this into an algorithm for testing one community versus two. Therefore, the ``hard'' regime for recovery contains the ``hard'' regime for testing community structure. We prove this in Theorem~\ref{thm:reduction_from1vs2} --- see Section~\ref{sec:reduction_from1vs2}. 

Although the reduction seems straightforward, there are some technical challenges: we suppose an algorithm recovers the planted community in the one-community case, but we cannot control how the algorithm behaves when there are two planted communities. 

Our reduction provides an alternative method for establishing detection-recovery gaps, rather than studying the recovery problem directly. For problems where recovery of the planted structure is strictly harder than detecting its presence, it is not viable to deduce optimal hardness of recovery from a planted-versus-null testing problem. However, we demonstrate that it is possible to attain the sharp recovery threshold via reduction from a planted-versus-planted problem, as long as the two planted distributions are appropriately chosen.

\subsection{Related Work and Open Problems}

\paragraph{The low-degree testing framework.}

Unfortunately, it seems to be beyond the current reach of computational complexity theory to prove that no polynomial-time algorithm can distinguish two random graph models, even under an assumption like $\mathrm{P} \ne \mathrm{NP}$. Nonetheless, a popular heuristic --- the \emph{low-degree testing framework}~\cite{pcal,hopkins-thesis,sos-detecting,HS-bayesian} (see~\cite{ld-notes} for a survey) --- gives us a rigorous basis on which to form conjectures about hardness of such problems. Specifically, we will study the power of \emph{low-degree tests}, a class of methods that includes tests based on edge counts, triangle counts, and other small subgraph counts. Strikingly, low-degree tests tend to be as powerful as all known polynomial-time algorithms for testing problems that are (informally speaking) of the flavor that we consider in this paper; see~\cite{hopkins-thesis,ld-notes} for discussion. In this paper, we will prove \emph{low-degree hardness}, meaning failure of all low-degree tests (to be defined formally in Section~\ref{sec:ld-testing}), for certain testing problems; this can be viewed as an apparent barrier to fast algorithms that we believe is unlikely to be overcome by known techniques, and perhaps indicates fundamental computational hardness.

\paragraph{Planted-versus-planted testing.}

We emphasize that there is a key difference between our work and existing hardness results for high-dimensional testing. The testing problems we consider are between two different ``planted'' distributions, each with a different type of planted structure. In contrast, previous low-degree hardness results for testing (e.g.,~\cite{hopkins-thesis,sos-detecting,HS-bayesian,ld-notes} and many others) have always considered testing between ``planted'' and ``null,'' where the null distribution has i.i.d.\ or at least independent entries. On a technical level, planted-versus-null problems are more tractable to analyze because we can explicitly construct a basis of orthogonal polynomials for the null distribution, which enables easy representation of the so-called `advantage' (see \eqref{eq:adv}), the main quantity to bound when using the low-degree polynomial approach, but this strategy seems more difficult to implement for planted-versus-planted problems.

The idea of planted-versus-null testing goes beyond the low-degree framework. Other forms of average-case lower bounds typically also, either explicitly or implicitly, leverage an easy-to-analyze null distribution; this includes reductions from planted clique (e.g.,~\cite{BR-reduction,BBH-reduction}), sum-of-squares lower bounds (e.g.,~\cite{pcal,KMOW}), and statistical query lower bounds (e.g.,~\cite{DKS-sq,FGRVX-sq}). In fact, these frameworks seem to struggle in settings where there is not a simple null distribution in the hypothesis test.

Our work overcomes this barrier that has limited the use of the above methods: we demonstrate for the first time that low-degree hardness results can be proven for planted-versus-planted problems. We give some general-purpose formulas (Propositions~\ref{prop:adv-gauss} and~\ref{prop:adv-binary}) that can be used to analyze a wide variety of such problems in random graphs or random matrices, not limited to just the specific models studied in this paper. The proof techniques are inspired by~\cite{SW-estimation}, which studies estimation problems rather than testing. On a technical level, the core challenge in our analysis is to bound certain recursively-defined quantities called $r_\alpha$ (defined in~\eqref{eq:r-def}). These are analogous to the cumulants that appear in~\cite{SW-estimation}, and while the $r_\alpha$ are not cumulants, they enjoy a number of similar convenient properties (see Section~\ref{sec:r_algebra}) that are important for the analysis.

\paragraph{Open problems.}

A natural next step is to investigate whether our method yields sharp computational thresholds for other problems that exhibit detection-recovery gaps. For example, the problem of parameter estimation in sparse high-dimensional linear regression likely has a detection-recovery gap (see~\cite{fp}) and can potentially be related to a testing problem between two planted models, e.g.\ between a sparse linear regression and a mixture of two sparse linear regressions.

Another open question is whether our computational hardness result can be shown in ways beyond the low-degree testing framework, such as by using the sum-of-squares framework, statistical query framework, or reduction from the planted clique problem. In particular, if the problem of testing community structure can be reduced from planted clique, this would yield a reduction from planted clique to planted dense subgraph \emph{recovery}, which is an open problem (see~\cite{BBH-reduction}).

\needspace{3\baselineskip}  
\section{Main results}

\subsection{Low-degree testing}
\label{sec:ld-testing}

We begin by explaining what it means for a low-degree test to distinguish two high-dimensional distributions.

\begin{definition}\label{def.separate}
Suppose $\PP_n$ and $\QQ_n$ are distributions on $\RR^N$ for some $N = N_n$. A \emph{degree-$D$ test} is a multivariate polynomial $f_n: \RR^N \to \RR$ of degree at most $D$ (really, a sequence of polynomials, one for each problem size $n$). Such a test $f$ is said to \emph{strongly separate} $\PP$ and $\QQ$ if, in the limit $n \to \infty$,
\[ \sqrt{\max\left\{\Var_\QQ[f], \Var_\PP[f]\right\}} = o\left(\left|\EE_\PP[f] - \EE_\QQ[f]\right|\right), \]
and \emph{weakly separate} $\PP$ and $\QQ$ if
\[ \sqrt{\max\left\{\Var_\QQ[f], \Var_\PP[f]\right\}} = O\left(\left|\EE_\PP[f] - \EE_\QQ[f]\right|\right). \]\end{definition}
Strong separation is a natural sufficient condition for success of a polynomial-based test because it implies (by Chebyshev's inequality) that $\PP$ and $\QQ$ can be distinguished by thresholding $f$'s output, with both type I and II errors $o(1)$. Weak separation also implies non-trivial testing, i.e.\ better than a random guess; see~\cite[Prop.\ 6.1]{fp}. In this paper, we characterize the limits of low-degree tests. For upper bounds, in the ``easy'' regime, we show that a constant-degree test achieves strong separation, implying a poly-time algorithm for testing with $o(1)$ error probability. For lower bounds, in the ``hard'' regime, we show that for some $D = \omega(\log n)$, no degree-$D$ test can achieve even weak separation. Because many known algorithms can be implemented as degree-$O(\log n)$ polynomials (e.g., spectral methods; see Section~4.2.3 of~\cite{ld-notes}), we treat this as ``evidence'' that no polynomial-time algorithm achieves non-trivial testing power, i.e.\ better than a random guess. Our results, in fact, often rule out much higher degree tests (e.g., $D = n^{\Omega(1)}$), depending on how far the parameters lie from the critical threshold.

\subsection{Model formulation}

We consider the problem of testing between two random graph models, both of which contain planted communities but with different community structures. We focus on testing between two \emph{additive Gaussian models} where the edge weights are Gaussian, and between two \emph{binary observation models} where the edges are unweighted and the diagonal is set to zero to ensure no self-loops in the graph. 

\begin{definition} [Additive Gaussian model]\label{def:Gaussian_comm}
Given the number of vertices $n$, total community size $k$, signal strength $\lambda>0$, number of communities $M$, and vector of community proportions $x \in [0,1]^{M}$ with $\sum_{\ell=1}^M x_\ell=1$, define the additive Gaussian model $\mathbb{P}=\mathbb{P}_{\rm Gaussian}(n,k,\lambda,M,x)$ as follows. 
Under $\mathbb{P}$, independently for each $i\in [n] := \{1,2,\ldots,n\}$, the community label $\sigma_i$ is sampled such that $\sigma_i = \ell$ with probability $x_\ell k/n$ for each $\ell \in [M]$ and $\sigma_i = \star$ (a symbol indicating membership in none of the communities) with probability $1-k/n$.
For each $i,j\in [n]$ with $i\leq j$, the edge weight $Y_{ij}$ is sampled from
\[
Y_{ij}\sim
\begin{cases}
\mathcal{N}\left(\frac{\lambda}{x_\ell}, 1\right), &\quad \sigma_i=\sigma_j=\ell\;\;\text{for some }\ell\in [M],\\
\mathcal{N}(0, 1), &\quad \text{otherwise}.
\end{cases}
\]
For $i>j$, the edge weight $Y_{ij}$ is defined to be $Y_{ji}$.\label{def:model1}\end{definition}
Notice that with the above definition, each community $\ell \in  \{1, 2, \ldots, M\}$ is expected to be of size $x_{\ell} k$ and the expected number of vertices which do not belong to any community is $n-k$. The choice of mean $\lambda/x_\ell$ ensures that on average, the vertices in one community have the same weighted degree (row sum of $Y$) as the vertices in any other community.

\begin{definition} [Binary observation model]
Given the number of vertices $n$, total community size $k$, edge probability parameters $q,s \ge 0$, number of communities $M$, and vector of community proportions $x\in \mathbb{R}^{M}$ with $\sum_{\ell=1}^M x_\ell=1$, define the Binary observation model $\mathbb{P}=\mathbb{P}_{\rm Binary}(n,k,q,s,M,x)$ as follows. 
The community labels $\{\sigma_i\}_{i \in [n]}$ are sampled the same way as in the additive Gaussian model. Given the community labels, for each pair of vertices $i,j\in [n]$ with $i < j$, the edge weight $Y_{ij}$ is sampled from
\[
Y_{ij}\sim
\begin{cases}
\mathrm{Bernoulli}\left(q+\frac{s}{x_\ell}\right), &\quad \sigma_i=\sigma_j=\ell\;\;\text{for some }\ell\in [M],\\
\mathrm{Bernoulli}\,(q), &\quad \text{otherwise}.
\end{cases}
\]
For $i>j$, the edge weight $Y_{ij}$ is defined to be $Y_{ji}$ and the diagonal entries set to zero $Y_{ii}=0$.\label{def:model2}\end{definition}
For example, if we want to model two communities  of equal sizes, we can choose $M=2$ and $x_1 = x_2 = \frac{1}{2}$. The communities are then both expected to be of size $k/2$. If we also set $s=q$ we have an in-community connection probability of $3q$ and every other pair of nodes is connected with probability $q$ as in the toy model discussed in the Introduction.

The two models introduced in Definitions~\ref{def:model1} and~\ref{def:model2} only differ in the edge weight distributions, as the community labels follow the same distribution under both models. Alternatively, we can write $S_\ell$ for the set of vertices in community $\ell$, so that $\sigma_i=\ell$ if and only if $i\in S_\ell$. Note that by definition, each vertex $i$ can belong to at most one community. In other words, the communities $\{S_\ell\}_{\ell \in [M]}$ are disjoint. 

With the other parameters fixed, we consider testing between model $\mathbb{P}$ with $M$ planted communities and community proportions $x \in [0,1]^M$, and the model $\mathbb{Q}$ with $M'$ planted communities and community proportions $x'\in [0,1]^{M'}$ for some $M'\neq M$. 
In short, for both Gaussian and Bernoulli edge weight models, we establish a ``hard'' regime where the distributions $\mathbb{P}$ and $\mathbb{Q}$ cannot be weakly separated by low-degree tests. We consider the regime $n \to \infty$ and allow all the parameters $k,\lambda,M,x$ to depend on $n$; thus, our results can apply to a growing number of communities, although our main focus is on the case where $M,M'$ are fixed so that our upper and lower bounds match.

\begin{theorem} [Additive Gaussian model]
\label{thm:Gaussian}
Given parameters $n, k, \lambda, M, M', x, x'$,
define distributions $\mathbb{P}=\mathbb{P}_{\rm Gaussian}(n,k,\lambda,M,x)$ and $\mathbb{Q}=\mathbb{P}_{\rm Gaussian}(n,k,\lambda, M',x')$.
Assume that $M \min_\ell x_\ell \ge C$ and $M' \min_\ell x'_\ell  \ge C$ for some constant $C>0$. Write $\widetilde{M}=|M-M'|$ and $\widehat{M}=\max\{M, M'\}$.
We have:
\begin{itemize}
\item If $D^5\widehat{M}^2\lambda^2(k^2/n\vee 1)=o(1)$, then no degree-$D$ test weakly separates $\mathbb{P}$ and $\mathbb{Q}$.
\item If $\widetilde{M}^2 \lambda^2 k^2/n=\omega(1)$ and $\widetilde{M}^2k/\widehat{M}^2=\omega(1)$, then there exists a degree-1 test that strongly separates $\mathbb{P}$ and $\mathbb{Q}$.
\end{itemize}
\end{theorem}

In the regime $k^2 \geq n$, $\widehat{M} = O(1)$, and $D \le \mathrm{polylog}(n)$, Theorem~\ref{thm:Gaussian} precisely characterizes (up to logarithmic factors) the computational threshold for low-degree testing. This threshold coincides with the conjectured computational threshold for \emph{recovering} a single planted community, which has been established in the low-degree polynomial framework~\cite[Theorem 2.5]{SW-estimation}. We focus on the $k^2\geq n$ regime in this paper, as this is where there is a conjectured detection-recovery gap, but we suspect that when $\widehat{M}$ is constant, $\lambda^2(k^2/n\vee 1)\sim 1$ is the computational threshold across the entire parameter regime. The optimal test when $k^2 < n$ should be based on the maximum diagonal entry, and while this is not a polynomial, it should be possible to approximate it by one (similar to Section~4.1.1 of~\cite{SW-estimation}).

\begin{theorem} [Binary observation model]
\label{thm:binary}
Given parameters $n, k, q, s, M, M', x, x'$, define distributions $\mathbb{P}=\mathbb{P}_{\rm Binary}(n,k,q,s,M,x)$ and $\mathbb{Q}=\mathbb{P}_{\rm Binary}(n,k,q,s, M',x')$. 
Assume that $M \min_\ell x_\ell \ge C$ and $M' \min_\ell x'_\ell  \ge C$ for some constant $C>0$ and that $q + s/(\min_\ell x_\ell) \le \tau_1$ for some constant $\tau_1 < 1$. Write $\widetilde{M}=|M-M'|$ and $\widehat{M}=\max\{M, M'\}$. We have:
\begin{itemize}
\item If $D^5 \widehat{M}^2 (s^2/q) (k^2/n\vee 1)=o(1)$, then no degree-$D$ test weakly separates $\mathbb{P}$ and $\mathbb{Q}$.
\item If $\widetilde{M}^{2/3} (s^2/q) k^2/n=\omega(1)$, $\widehat{M}^{-1/3}sk=\omega(1)$ and $\widetilde{M}^2k/\widehat{M}^2=\omega(1)$ then there exists a degree-3 test that strongly separates $\mathbb{P}$ and $\mathbb{Q}$.
\end{itemize}
\end{theorem}

The upper and lower bounds match (up to log factors) provided $k^2 \ge n$, $\widehat{M} = O(1)$, $D \le \mathrm{polylog}(n)$, and $q \ge 1/n$. The condition $q \ge 1/n$ is natural since without it there will be isolated vertices. The regime $k^2 < n$ is more complicated, and some open questions remain here even for simpler testing and recovery problems than those we study here; see Section~2.4.1 of~\cite{SW-estimation} for discussion.

\subsection{Formal relation to low-degree recovery}\label{sec:formal_equiv}


While there has been much work on studying the low degree polynomial framework for measuring the computational hardness of problems when one wishes to \emph{detect} the presence of hidden structures, a recent line of work~\cite{SW-estimation} has extended the framework to the setting where one wishes to \emph{recover} a planted structure from noisy data. We use this recent work to show a connection between the general low degree frameworks for detection and recovery. Namely, for any given recovery problem where one wishes to estimate some scalar quantity $g(X) \in \mathbb{R}$ of the planted distribution (e.g.\ $g(X)=X_1$) from some observation $Y \in \mathbb{R}^N$, 
there is a specific testing problem with the same computational hardness regimes and, in the other direction, for any testing problem $H_0: Y \sim \mathbb{P}$ versus $H_1: Y \sim \mathbb{Q}$, as long as $\mathbb{P}_X \ll \mathbb{Q}_X$ so that the likelihood ratio of the signal $d \mathbb{P}_X/ d \mathbb{Q}_X$ exists, there is a corresponding recovery problem with the same computational hardness regimes.

To show these equivalences, we compare the object $\Corr$ from~\cite[Eq.~(2)]{SW-estimation} to our notion of advantage $\Adv$. In particular, we have
\begin{equation}
\Corr_{\leq D}(g(X),\widetilde{\QQ}) := \sup_{f \text{ deg } D 
}  \frac{\E_{\widetilde{\QQ}}[g(X) f(Y)]}{ \sqrt{\E_{\widetilde{\QQ}}[f(Y)^2] }},
\end{equation}
where we have written $g(X)$ to indicate that the scalar quantity we are estimating depends on the underlying unknown signal $X$, and we treat $\widetilde{\QQ}$ as the joint distribution of $(X, Y)$. We also define $\NCorr$ that renormalizes $\Corr$ by dividing through by $\E_{\widetilde{\QQ}}[g(X)]$:
\begin{equation}\label{eq:NCorr}
\NCorr_{\leq D}(g(X),\widetilde{\QQ}) := \frac{1}{\E_{\widetilde{\QQ}}[g(X)]} \sup_{f \text{ deg } D 
} \frac{\E_{\widetilde{\QQ}}[g(X) f(Y)]}{ \sqrt{\E_{\widetilde{\QQ}}[f(Y)^2] }}.
\end{equation}
We compare $\NCorr$ to our advantage $\Adv$ defined in \eqref{eq:adv},
\begin{equation}\label{eq:adv_repeated}
\Adv_{\le D}(\PP,\QQ) := \sup_{f \text{ deg } D} \frac{\E_\PP[f(Y)]}{\sqrt{\E_\QQ[f(Y)^2]}},
\end{equation}
and, loosely, we notice that $\Adv_{\le D}(\PP,\QQ)$ equals $\NCorr_{\leq D}(g(X),\widetilde{\QQ})$ when $\QQ$ equals $\widetilde{\QQ}$ and $\PP$ is defined such that $\E_\PP[f(Y)] = \E_{\widetilde{\QQ}}[g(X) f(Y)]/{\E_{\widetilde{\QQ}}[g(X)]}$. This gives rise to the following proposition.

In the testing and recovery problems respectively, the conjectured-hard region, i.e.\ where low-degree algorithms fail, is that where $\Adv$ and $\Corr'$ respectively are {$1+o(1)$ which for testing rules out $f$ which weakly distinguish $f$ in the testing problem (and $O(1)$ which rules out $f$ which strongly distinguish).  

For the case of recovery the notion of failure relates to the low-degree minimum mean squared error ${\rm MMSE}_{\leq D}:=\inf_{f \text{ deg } D}\mathbb{E}_{\widetilde{\QQ}}[(f(Y) - g(X))^2]$. Note the trivial estimator $f(Y)=\mathbb{E}_{\widetilde{\QQ}}[g(X)]$, a constant polynomial, achieves ${\rm MSE}$ of $\mathbb{E}_{\widetilde{\QQ}}[g(X)^2] - \mathbb{E}_{\widetilde{\QQ}}[g(X)]^2$ and by Fact~1.1 of~\cite{SW-estimation},
\[  
{\rm MMSE}_{\leq D} = \E_{\widetilde{\QQ}}[g(X)^2] - \E_{\widetilde{\QQ}}[g(X)]^2 \NCorr^2_{\leq D}(g(X),\widetilde{\QQ}).
\]
Thus $\NCorr=1+o(1)$ rules out $f$ which do substantially better than the trivial estimator.

%
\begin{proposition}\label{prop.equivs} Suppose all probability distributions are either discrete or have a density function.
    \phantom{ } 
    \begin{enumerate}
        \item \label{recovery_has_testing_equiv}
        Given any recovery problem to estimate non-negative $g(X)$ given $Y$ with joint distribution $(X,Y)\sim \QQ$ and $0<\E_\QQ[g(X)]<\infty$, there exists a joint distribution $(X,Y) \sim \PP$ such that for the planted testing problem 
        $H_0: \; Y \sim \QQ$ vs $H_1:\; Y \sim \PP$,
        \[ \Adv_{\le D}(\PP,\QQ) =  \NCorr_{\leq D}(g(X),\QQ). \]
        \item\label{testing_with_likelihood_ratio_has_recovery}
        Given a testing problem $H_0: Y\sim \QQ$ and $H_1: Y\sim \PP$ between two planted distributions $ (X,Y)\sim \QQ$ and $(X,Y)\sim \PP$, if $\ell(X):={d\PP_X}/{d\QQ_X}$ exists and $Y|X$ is the same under $\PP$ and~$\QQ$, then for the estimation problem $g(X)=\ell(X)$ in $\QQ$,        
        %
        \[ \Adv_{\le D}(\PP,\QQ) =  \NCorr_{\leq D}(g(X), \QQ). \]
        %
    \end{enumerate}
\end{proposition}

The proof of Proposition~\ref{prop.equivs} is included in Section~\ref{sec:recoversSWrecovery}, where we also show that our Propositions~\ref{prop:adv-gauss} and~\ref{prop:adv-binary} generalize the cumulant upper bounds on $\Corr$ shown in~\cite{SW-estimation}.

Consider the important special case when we estimate the indicator of an event $A$ on the signal, meaning that $g(X) = \mathbb{I}\{A\}$. In this case, Proposition~\ref{prop.equivs} shows that the equivalent testing problem is one where we test between $H_1: X \sim \mathbb{P}$ versus $H_0: X \sim \mathbb{Q}$ where $\mathbb{Q}$ is the measure describing the randomness in the observed $Y$, and $\mathbb{P}$ is the measure of $Y$ conditional on the occurrence of event~$A$.

In particular, for recovery in the planted dense submatrix (PDS) model one estimates the indicator that $v_1$ is in the single planted community -- the problem originally considered in~\cite{SW-estimation}. The equivalent testing problem is to test between PDS and PDS conditioned on the event that $v_1$ is in the planted community.

Note that for the particular testing problem considered in the rest of this paper, that of testing between $M$ and $M'\neq M$ communities the proposition does not clearly give an equivalent recovery problem. (For signal $X$ and observation $Y$ as in Definition~\ref{def:Gaussian_comm} the supports of $\PP_X$ and $\QQ_X$ are disjoint.)

\needspace{3\baselineskip}
\subsection{Proof overview}\label{sec:proof.overview}

\paragraph{Main quantity to bound: advantage.}

In order to rule out weak separation between distributions $\PP = \PP_n$ and $\QQ = \QQ_n$ on $\RR^{N_n}$, it will suffice to bound the degree-$D$ ``advantage,'' named as such to emphasize that it measures the ability of low-degree polynomials to outperform random guessing. The degree-$D$ advantage, $\Adv_{\le D}$, is defined as
\begin{equation}
\label{eq:adv}
\Adv_{\le D}(\PP,\QQ) := \sup_{f \text{ deg } D} \frac{\E_\PP[f]}{\sqrt{\E_\QQ[f^2]}},
\end{equation}
where $f$ ranges over polynomials $\RR^N \to \RR$ of degree at most $D$. The quantity $\Adv_{\le D}$ is also the \emph{norm of the degree-$D$ likelihood ratio} (see~\cite{hopkins-thesis,ld-notes}), but we will not use this interpretation here as the likelihood ratio is difficult to work with in our setting. We note that while the notion of separation is symmetric between $\PP$ and $\QQ$, the notion of advantage is not; for our purposes, we could just as easily work with $\Adv_{\le D}(\QQ,\PP)$ instead of $\Adv_{\le D}(\PP,\QQ)$. The following basic fact connects $\Adv_{\le D}$ with strong/weak separation.

\begin{lemma}\label{lem:adv-sep}
Fix a sequence $D = D_n$.
\begin{itemize}
\item If $\Adv_{\le D}(\PP,\QQ) = O(1)$ then no degree-$D$ test strongly separates $\PP$ and $\QQ$.
\item If $\Adv_{\le D}(\PP,\QQ) = 1+o(1)$ then no degree-$D$ test weakly separates $\PP$ and $\QQ$.
\end{itemize}
\end{lemma}

\noindent The proof of Lemma~\ref{lem:adv-sep}, along with the proofs of all facts in this section, can be found in Section~\ref{sec:additional-proofs}. In light of Lemma~\ref{lem:adv-sep}, it remains to bound $\Adv_{\le D}$. We will provide a few general-purpose bounds, one for Gaussian problems and one for binary-valued problems. Both will involve recursively-defined quantities $r_\alpha$, introduced in what follows.

\paragraph{Recursive definition for $r_\alpha$.}

Suppose $X$ is a random variable taking values in $\RR^N$, which may have a different distribution under $\PP$ and $\QQ$. For $\alpha,\beta \in \NN^N$ where $\NN = \{0,1,2,\ldots\}$, define 
$$|\alpha| := \sum_i \alpha_i, \qquad \alpha! := \prod_i \alpha_i!, \qquad \binom{\alpha}{\beta} := \prod_i \binom{\alpha_i}{\beta_i}, \qquad \text{ and } \qquad X^\alpha := \prod_i X_i^{\alpha_i}.$$ 
Also define $\beta \le \alpha$ to mean ``$\beta_i \le \alpha_i$ for all $i$'' and define $\beta \lneq \alpha$ to mean ``$\beta_i \le \alpha_i$ for all $i$ and for some $i$ the inequality is strict: $\beta_i < \alpha_i$.'' With this notation in hand, define $r_\alpha = r_\alpha(X) \in \RR$ for $\alpha \in \NN^N$ recursively by
\begin{equation}\label{eq:r-def}
r_\alpha = \EE_\mathbb{P}\left[X^\alpha\right] - \sum_{0 \le \beta\lneq \alpha}r_\beta \, \binom{\alpha}{\beta} \, \EE_\mathbb{Q}\left[X^{\alpha-\beta}\right].
\end{equation}

\paragraph{Bounds on advantage.}

We have the following general-purpose bounds on $\Adv_{\le D}$ in terms of $r_\alpha$ defined in \eqref{eq:r-def}. The proofs are inspired by~\cite{SW-estimation} and can be found in Section~\ref{sec:additional-proofs}.

\begin{proposition}[General additive Gaussian model]
\label{prop:adv-gauss}
Suppose $\PP$ and $\QQ$ take the following form: to sample $Y \sim \PP$ (or $Y \sim \QQ$, respectively), first sample $X \in \RR^N$ from an arbitrary prior $\PP_X$ (or $\QQ_X$, resp.), then sample $Z \sim \mathcal{N}(0,I_N)$, and set $Y = X + Z$. Define $r_\alpha$ as in~\eqref{eq:r-def}. Then
\[ \Adv_{\le D}(\PP,\QQ) \le \sqrt{\sum_{\alpha \in \NN^N, \, |\alpha| \le D} \frac{r_\alpha^2}{\alpha!}}. \]
\end{proposition}

\begin{proposition}[General binary observation model]
\label{prop:adv-binary}
Suppose $\PP$ and $\QQ$ each take the following form. To sample $Y \sim \PP$ (or $Y \sim \QQ$, respectively), first sample $X \in \RR^N$ from an arbitrary prior $\PP_X$ (or $\QQ_X$, resp.) supported on $X \in [\tau_0,\tau_1]^N$ with $0 < \tau_0 \le \tau_1 < 1$, then sample $Y \in \{0,1\}^N$ with entries conditionally independent given $X$ and $\EE[Y_i | X] = X_i$. Define $r_\alpha$ as in~\eqref{eq:r-def}. Then
\[ \Adv_{\le D}(\PP,\QQ) \le \sqrt{\sum_{\alpha \in \{0,1\}^N, \, |\alpha| \le D} \frac{r_\alpha^2}{(\tau_0(1-\tau_1))^{|\alpha|}}}. \]
\end{proposition}

\paragraph{Combinatorial properties of the $r_\alpha$.}
The upshot of the two propositions above is that to show hardness of distinguishing $\mathbb{P}$ versus $\mathbb{Q}$, it suffices to bound the recursively defined $r_\alpha$. This task is made easier by indentifying combinatorial properties that the $r_\alpha$ enjoy. In Section~\ref{sec:r_algebra}, we show general results for how properties of the probability spaces transfer to behaviour of the $r_\alpha$.
We may consider $\alpha \in \mathbb{N}^N$ as a multigraph (see Section~\ref{subsec:alpha_as_graph}), and we word the results in this language. Loosely speaking, the results we present are as follows:
\begin{itemize}
\item If $\mathbb{P}$ and $\mathbb{Q}$ are multiplicative for disjoint graphs $\alpha$ and $\beta$ then $r_{\alpha\cup\beta}=r_\alpha r_\beta$ (Lemma~\ref{lem:multiply}).

\item If $\mathbb{E}_\mathbb{P}[X^\tau]=\mathbb{E}_\mathbb{Q}[X^\tau]$ for all trees $\tau$, then the $r$ value is zero on trees (Lemma~\ref{lem:forest}).

\item The $r$-values are indifferent to constant shifts to $X$ (Lemma~\ref{lem:constantshift}).

\item If we scale $X$ by a constant factor $c$ to construct $\widetilde{r}$, then $\widetilde{r}_\alpha=c^{|\alpha|}r_\alpha$ (Lemma~\ref{lem:scalar}).
\end{itemize}

\paragraph{Putting it all together.}

In Section~\ref{sec:proof}, we pivot back to considering our particular probability spaces $\mathbb{P}$ and $\mathbb{Q}$ and calculate the expected value of $X^\alpha$ as a function of properties of the graph $\alpha$ (see Lemma~\ref{lem:integralgenx}). This, together with the multiplicative and tree results for the $r_\alpha$, allow us to bound  $\Adv$ in the Gaussian case. The scaling and shifting properties of the $r_\alpha$ are used to show we can deduce the graph case from the Gaussian case.

\subsection{New contributions in the extended version of this paper.}\label{sec:new}
This is an extended version of a paper accepted to Innovations in Theoretical Computer Science (ITCS). In that paper, we introduced a low-degree framework for testing \emph{between two planted models} and, in this extended version of the paper, we establish a formal connection of that framework to low degree \emph{recovery}. See Sections~\ref{sec:formal_equiv} and~\ref{sec:recoversSWrecovery} for discussion and proofs, respectively. Proposition~\ref{prop.equivs} establishes that any low-degree recovery problem is formally equivalent to a low-degree problem of testing between two planted models. Thus, our framework for testing between planted distributions extends that of recovery. 

The non-recursive definition for the $r_\alpha$ in~\eqref{eq:r_non_rec} is also new, as well as their relation to the $\kappa_\alpha$ introduced in~\cite{SW-estimation} -- see Remark~\ref{rem:recoverSWrecovery}. 

In a very different connection of our results to recovery, in Theorem~\ref{thm:reduction_from1vs2}, we prove a reduction from weak recovery for a planted dense submatrix (PDS) to testing between the 1 or 2 community models. The reduction holds for problem parameters under which we proved hardness for testing between 1 and~2 communities; hence, giving an alternate indication of hardness for weak recovery in the PDS problem. Recall that a well established detection--recovery gap has been shown for the PDS model. Therefore, a reduction to the testing problem between a PDS and the standard null distribution (with independent mean zero Gaussian entries) is not sufficient to show hardness.

\needspace{3\baselineskip}
\section{The algebra of planted vs planted}
\label{sec:r_algebra}

The recursively defined $r_\alpha$ play a central role in our proof. By  Proposition~\ref{prop:adv-gauss}, the `advantage' $\Adv$ is bounded above by a sum of squares of the $r_\alpha$; therefore, to show low-degree hardness for a distinguishing problem, it is enough to control the size of this sum of squares. In this section, we explore the combinatorial behaviour of these $r_\alpha$ and show that they exhibit very nice properties under only mild assumptions on the probability spaces $P$ and $Q$. (We will write $P$ and $Q$ for the probability spaces in this section, both to emphasize that these results hold for general probability spaces and to ease the notational burden.) We assume throughout that both $P$ and $Q$ are symmetric, i.e.\ they are supported on $X$ for which $X_{ij} = X_{ji}$.

\subsection{The graph interpretation}\label{subsec:alpha_as_graph}
As in \cite{SW-estimation}, it will be convenient to think of $\alpha \in \mathbb{N}^N$ as a multigraph, possibly with self-loops, on the vertex set $[n]$, where $N = n(n+1)/2$ and we take $\alpha_{ij}$, for each $i \leq j$, to be the number of edges between vertices $i$ and $j$. For example, for $n=3$, $N=6$ and if we fix the order to be $\alpha=(\alpha_{11}, \alpha_{12}, \alpha_{22}, \alpha_{13}, \alpha_{23}, \alpha_{33})$ then $(0,2,0,1,1,0)$ is the graph $\dtrian$ and, for $n=2$, $N=3$, $(1,1,0)$ is the graph $\dlopsidedlollipop$ a single edge with a loop at one vertex.
For graphs $\alpha$ and~$\beta$, we consider $\beta$ to be a subgraph of $\alpha$, denoted $\beta \subseteq \alpha$, if the labelled edge set of $\beta$ is a subset of the labelled edge set of $\alpha$. For example, for the graph $\alpha=\dtrian$, the graphs $\beta=\dcherrya$ and $\beta'=\dcherryb$ are distinct subgraphs. For graph $\alpha$ and subgraph $\beta\subseteq \alpha$, define $\alpha \backslash \beta$ to be the graph obtained from $\alpha$ by first removing the labelled edges in $\beta$ and then removing isolated vertices - e.g.\ $\trian \backslash \cherry = \hedge$ (not $\dumbedge$). Similarly, let $\alpha \cap \beta$ denote the graph obtained by first taking the intersection of both graphs and then deleting any isolated vertices. 

The usefulness of considering graphs with labelled edge sets is that it simplifies the expression for the recursion in the definition of $r_\alpha$. To avoid confusion, write $v_\alpha$ for the vector that maps to the graph $\alpha$. Note, first, that if $v_\beta \leq v_\alpha$, then $\beta \subseteq \alpha$ and vice-versa. However, the counts are different. For fixed $\alpha$ and $\beta$, with $v_\beta \leq v_\alpha$, there are $\binom{v_\alpha}{v_\beta}$ many distinct edge-labelled graphs $\beta'$ such that $\beta'\subseteq \alpha$ and $v_{\beta'}=v_\beta$.
Hence, for edge-labelled graphs, the equivalent recursive definition to \eqref{eq:r-def} is as follows, where the sum is over edge-labelled subgraphs. 
For graph $\alpha$, the term $r_\alpha$ is defined recursively by 
\begin{equation}\label{eq:r_def_graphs}
r_\alpha = \E_P\left[X^\alpha\right] - \sum_{\varnothing \subseteq \beta \subsetneq \alpha} r_{\beta} \, \E_Q \left[X^{\alpha \backslash \beta}\right],
\end{equation}
starting from the base case of the empty graph $r_\varnothing=1$.
For example, if the probability spaces are exchangeable (i.e.\ $\beta=\cherry$ and $\beta'=\sidecherry$ etc.\ have the same expectations under both $P$ and $Q$) then 
\[ r_{\trian} = \E_P\left[X^{\trian}\right] - \E_Q\left[X^{\trian}\right]  -3r_{\edge}\, \E_Q\left[X^{\cherry}\right] -3r_{\cherry}\, \E_Q\left[X^{\edge}\right].\] 
{We may also define the $r_\alpha$ non-recursively as follows --- see Remark~\ref{rem:non_recursive} and the end of this section where we show the equivalence of the two definitions. For any set $\alpha$ let $\mathcal{P}(\alpha)$ be the set of partitions of $\alpha$. For partition $\tau$ let $|\tau|$ denote the number of parts and $\gamma \in \tau$ indicates that the set $\gamma$ is a part in the partition~$\tau$. (For example $\mathcal{P}(\{1,2,3\})$ is the set with elements $\{1,2,3\}$, $\{1\}\{2,3\}$,  $\{2\}\{1,3\}$,  $\{1,2\}\{3\}$ and $\{1\}\{2\}\{3\}$ and if $\tau=\{2\}\{1,3\}$ then $|\tau|=2$ and $\{2\},\{1,3\} \in \tau$.)
\begin{equation}\label{eq:r_non_rec} r_\alpha 
= \sum_{\varnothing \subseteq \delta \subseteq \alpha} \E_{P}[X^\delta] \sum_{\tau \in \mathcal{P}(\alpha\backslash \delta)} (-1)^{|\tau|} \;|\tau|! \prod_{\gamma \in \tau}\E_Q[X^{\gamma}] 
\end{equation}
}

We will use the notion of edge-labelled subgraphs, denoted $\subseteq$, to aid the proofs, but for the rest of the paper we consider $\alpha \in \mathbb{N}^N$, or equivalently $\alpha$ a graph without edge labels, and denote by~$\leq$ the non-labelled subset or subgraph relation.

\subsection{Combinatorial properties of the $r$-values} 
We will be interested in how properties of the probability spaces transfer to the behaviour of $r_\alpha$. We will see that the $r_\alpha$ behave multiplicatively over taking disjoint unions if the following property holds for $P$ and $Q$. Let $\alpha \cup \beta$ denote the disjoint union of $\alpha$ and $\beta$. We say the probability space $A$ is multiplicative over disjoint unions if \eqref{disconnect_indep} holds:
\begin{equation}\label{disconnect_indep}\E_A\left[X^{\alpha\cup\beta}\right]= \E_A\left[X^{\alpha}\right]\E_A\left[X^{\beta}\right] \mbox{ for any graphs $\alpha$ and $\beta$.}
\end{equation} 

\begin{lemma}\label{lem:multiply} Suppose $P$ and $Q$ are symmetric and multiplicative over disjoint unions, i.e.\ they satisfy \eqref{disconnect_indep}, then for any $\alpha$ and $\beta$, we have $r_{\alpha \cup \beta}=r_\alpha \, r_\beta.$
\end{lemma}

\begin{proof} We proceed by induction on the number of edges. \noindent\emph{Base case.} Suppose $\alpha$ consists of two disjoint edges, denoted by $\edge$ and $\edgedotted$. Then from \eqref{eq:r_def_graphs}, \[r_{\edge  \; \edgedotted} = \E_P\left[X^{\edge}\right]\E_P\left[X^{\edgedotted}\right]-r_\varnothing \, \E_Q\left[X^{\edge}\right]\E_Q\left[X^{\edgedotted}\right]-r_{\edge} \, \E_Q\left[X^{\edgedotted}\right]-r_{\edgedotted} \, \E_Q\left[X^{\edge}\right], \] 
where we used the multiplicative property of \eqref{disconnect_indep} to deduce $\E_P\left[X^{\edge \; \edgedotted}\right]=\E_P\left[X^{\edge}\right]\E_P\left[X^{\edgedotted}\right]$ and, similarly, $\E_Q\left[X^{\edge \; \edgedotted}\right]=\E_Q\left[X^{\edge}\right]\E_Q\left[X^{\edgedotted}\right]$. Now substituting $r_\varnothing=1$ along with $r_{\edge}=\E_P\left[X^{\edge}\right]-\E_Q\left[X^{\edge}\right]$ and the corresponding expression for $r_{\edgedotted}$, we get \[r_{\edge  \; \edgedotted} = \E_P\left[X^{\edge}\right]\E_P\left[X^{\edgedotted}\right]-\E_P\left[X^{\edge}\right] \E_Q\left[X^{\edgedotted}\right]-\E_P\left[X^{\edgedotted}\right] \E_Q\left[X^{\edge}\right]+\E_Q\left[X^{\edge}\right]\E_Q\left[X^{\edgedotted}\right]=r_{\edge} \, r_{\edgedotted}.\]
\noindent\emph{Inductive step.} Fix $\tau=\alpha\cup\beta$ (where $\alpha, \beta$ aare disjoint) and assume the factorization of $r$ holds for graphs with fewer than $|\tau|=|\alpha|+|\beta|$ edges. For any graph $\gamma$ define $z_\gamma$ by $z_\gamma:=\E_P\left[X^\gamma\right]- r_\gamma$. Then, first note
\[ z_\alpha \, z_\beta = \E_P\left[X^\alpha\right]\E_P\left[X^\beta\right] - r_\alpha \, \E_P\left[X^\beta\right] - r_\beta\, \E_P\left[X^\alpha \right] + r_\alpha \, r_\beta, \]
and that because $\alpha$ and $\beta$ are disjoint and $P$ satisfies \eqref{disconnect_indep}, we have $\E_P\left[X^\alpha\right]\E_P\left[X^\beta\right]=\E_P\left[X^{\alpha \cup \beta}\right]$. Hence,
\[ z_\alpha \, z_\beta = \E_P\left[X^{\alpha \cup \beta}\right] - r_\alpha \, \E_P\left[X^\beta\right] - r_\beta \, \E_P\left[X^\alpha \right] + r_\alpha \,r_\beta, \]
and now (back) substituting $\E_P\left[X^\alpha\right]=r_\alpha+z_\alpha$ and $\E_P\left[X^\beta\right]=r_\beta+z_\beta$ we find
\[ r_\alpha \, r_\beta= \E_P\left[X^{\alpha \cup \beta}\right]-r_\alpha \, z_\beta - r_\beta \, z_\alpha - z_\alpha \, z_\beta. \]
By definition, $r_{\alpha \cup \beta}=\E_P\left[X^{\alpha \cup \beta}\right]-z_{\alpha \cup \beta}$; therefore, to complete the proof, it suffices to show the identity 
\begin{equation}\label{eq:identity}
    z_{\alpha \cup \beta} = z_\beta \, r_\alpha +z_\alpha \, r_\beta +z_\alpha \, z_\beta.
\end{equation}
Again, write $\tau=\alpha \cup \beta$ and note that by the definitions of $r_\tau$ and $z_\tau$, we have
\begin{equation}\label{eq:zG}
z_\tau=  \E_P\left[X^\tau\right] - r_\tau = \sum_{\varnothing \subseteq \gamma \subsetneq \tau} r_{\gamma} \E_Q\left[X^{\tau\backslash \gamma}\right]. 
\end{equation}
Now observe that for any $\gamma \subsetneq \tau$, because $\tau=\alpha \cup \beta$, we have $\gamma=\gamma_\alpha \cup \gamma_\beta$ where $\gamma_\alpha=\gamma \cap \alpha$ and $\gamma_\beta=\gamma \cap \beta$. Thus, $\gamma$ is a disjoint union of $\gamma_\alpha$ and $\gamma_\beta$ with strictly fewer total edges than $\alpha \cup \beta$. Therefore, $(\alpha \cup \beta) \backslash \gamma = (\alpha \backslash \gamma_\alpha) \cup (\beta \backslash \gamma_\beta)$ where $\alpha \backslash \gamma_\alpha$ and $\beta \backslash \gamma_\beta$ are disjoint, so by assumption, $\E_Q\left[X^{(\alpha \cup \beta) \backslash \gamma}\right] = \E_Q\left[X^{\alpha \backslash \gamma_\alpha}\right] \E_Q\left[X^{\beta \backslash \gamma_\beta}\right]$ and  by the inductive hypothesis $r_{\gamma}=r_{\gamma_\alpha}r_{\gamma_\beta}$. Hence, for any fixed $\gamma \subsetneq \alpha \cup \beta$,
\begin{equation}\label{eq:w1} r_\gamma \, \E_Q\left[X^{(\alpha \cup \beta) \backslash \gamma}\right] = r_{\gamma_\alpha} \, \E_Q\left[X^{\alpha \backslash \gamma_\alpha}\right] \, r_{\gamma_\beta} \, \E_Q\left[X^{\beta \backslash \gamma_\beta}\right].
\end{equation}
There are two special cases for \eqref{eq:w1}. If $\gamma_\alpha=\alpha$, then 
\begin{equation}
\label{eq:w1a}
r_{\gamma} \, \E_Q\left[X^{(\alpha \cup \beta)\backslash \gamma}\right]=r_{\alpha}\, r_{\gamma_\beta} \,\E_Q\left[X^{\beta \backslash \gamma_\beta}\right],
\end{equation}
and symmetrically for the case $\gamma_\beta=\beta$. Note that because $\gamma$ is a strict subgraph of $\alpha \cup \beta$ either of $\gamma_\alpha=\alpha$ or $\gamma_\beta=\beta$ may hold but not both. 

In the expression for $z_{\tau} = z_{\alpha \cup \beta}$ in \eqref{eq:zG} we take the sum over $\{ \gamma : \varnothing \subseteq \gamma \subsetneq \alpha \cup \beta \}$  and partition it into sums over the sets $S_1=\{ \gamma : \gamma_\alpha = \alpha,\; \varnothing \subseteq \gamma_\beta \subsetneq \beta \}$, $S_2=\{ \gamma : \varnothing \subseteq \gamma_\alpha \subsetneq \alpha,\; \gamma_\beta = \beta \}$ and $S_3=\{ \gamma : \varnothing \subseteq \gamma_\alpha \subsetneq \alpha, \; \varnothing \subseteq \gamma_\beta \subsetneq \beta \}$. We begin with $S_1$. By \eqref{eq:w1a},
\begin{equation}
\label{eq:razb}
\sum_{\gamma \in S_1} r_{\gamma} \, \E_Q\left[X^{(\alpha \cup \beta)\backslash \gamma}\right] = \sum_{ \varnothing \subseteq \gamma_\beta \subsetneq \beta } r_\alpha\,  r_{\gamma_\beta} \, \E_Q\left[X^{\beta \backslash \gamma_\beta}\right] = r_\alpha \, z_\beta.
\end{equation}
The final step uses that $z_\beta = \sum_{\varnothing \subseteq \gamma_\beta \subsetneq \beta } r_{\gamma_\beta} \, \E_Q\left[X^{\beta \backslash \gamma_\beta}\right]$, similar to \eqref{eq:zG}.
Next, taking the sum over $S_2$ yields $r_\beta \, z_\alpha$ in the same way as in \eqref{eq:razb}. Lastly, the sum over $S_3$ is given by
\begin{equation}
\label{eq:zazb}
\sum_{\gamma \in S_3} r_{\gamma } \, \E_Q\left[X^{(\alpha \cup \beta)\backslash \gamma}\right] = \sum_{\varnothing \subseteq \gamma_\alpha \subsetneq \alpha, \, \varnothing \subseteq \gamma_\beta \subsetneq \beta}  r_{\gamma_\alpha} \, \E_Q\left[X^{\alpha \backslash \gamma_\alpha}\right] \, r_{\gamma_\beta} \, \E_Q\left[X^{\beta \backslash \gamma_\beta}\right] = z_\alpha \, z_\beta.
\end{equation}
By \eqref{eq:zG}, we see that $z_{\tau} = z_{\alpha \cup \beta}$ can be obtained as the following sum:
\begin{align*}
z_{\alpha \cup \beta} = z_\tau= \sum_{\varnothing \subseteq \gamma \subsetneq \tau} r_{\gamma} \E_Q\left[X^{\tau\backslash \gamma}\right] &= \sum_{\gamma \in S_1} r_{\gamma} \E_Q\left[X^{\tau\backslash \gamma}\right] + \sum_{\gamma \in S_2} r_{\gamma} \E_Q\left[X^{\tau\backslash \gamma}\right] + \sum_{\gamma \in S_3} r_{\gamma} \E_Q\left[X^{\tau\backslash \gamma}\right] \\
&= r_\alpha \, z_\beta+ r_\beta \, z_\alpha+z_\alpha \, z_\beta.
\end{align*}
where in the final step above we have used results \eqref{eq:razb}-\eqref{eq:zazb}. Thus, we find \eqref{eq:identity}, confirming the identity as required.
\end{proof}

\begin{lemma}\label{lem:forest} For all $\tau$ where $\tau$ is a forest, meaning a graph with no cycles, suppose that $P$ and $Q$ satisfy $\E_P\left[X^\tau\right]=\E_Q\left[X^\tau\right].$ Then, $r_\alpha=0$ for any forest graph $\alpha$.\end{lemma}

\begin{proof}
The proof is almost immediate by induction on the number of edges. For the base case we note $r_{\edge}=  \E_P\left[X^{\edge}\right]-\E_Q\left[X^{\edge}\right]=0$. For any fixed forest $\alpha$ and $\beta \subsetneq \alpha$, the graph $\beta$ is a forest on strictly fewer edges and so by induction $r_{\beta}=0$, but then $r_\alpha=\E_P\left[X^\alpha\right]-\E_Q\left[X^\alpha\right]=0$. 
\end{proof}

We also show that one can add a constant shift to the distribution without changing the values of the $r_\alpha$. The proof is somewhat technical, so we relegate it to Section~\ref{sec:additional-proofs}.
\begin{lemma}\label{lem:constantshift} 
Let $\widetilde{X}$ be defined by $\widetilde{X}_{ij}=X_{ij}+y_{ij}$ where $y_{ij} \in \mathbb{R}$ is non-random for each pair $i,j$. For any probability spaces $P$ and $Q$, let $r_\alpha=r_\alpha(P,Q,X)$ and $\widetilde{r}_\beta=\widetilde{r}_\beta(P,Q,\widetilde{X})$. Then, for all $\gamma$, we have that $r_\gamma=\widetilde{r}_\gamma.$
\end{lemma}
The following lemma concerns the effect on $r$ of scaling.
\begin{lemma}\label{lem:scalar} Fix $a \in \mathbb{R}$ and $a \neq 0$. Let $\widetilde{X}$ be defined by $\widetilde{X}_{ij}=aX_{ij}$. Then for any probability spaces $P$ and $Q$, for $r_\alpha=r_\alpha(P,Q,X)$ and $\widetilde{r}_\beta=\widetilde{r}_\beta(P,Q,\widetilde{X})$, and for all $\gamma$, we have that $\widetilde{r}_\gamma = a^{|\gamma|} \, r_\gamma,$
where $|\gamma|$ equals the number of edges in the graph $\gamma$, i.e.\ $|\gamma|= |E(\gamma)|$.
\end{lemma}

\begin{proof}
This proof is a simple induction on $|\alpha|$. The base case is easy as $\tr_\varnothing=r_\varnothing=1$ as required. Now, fix $\alpha$ for some $|\alpha|>1$, and assume we have proven the result for $|\beta| < |\alpha|$. Notice, that by definition, we have $\E_P\left[\tX^\alpha\right] =  a^{|\alpha|}\E_P\left[X^\alpha\right]$ and $\E_Q\left[\tX^{\beta}\right] =  a^{|\beta|} \E_Q\left[X^{\beta}\right]$. Therefore,
\[
\tr_\alpha= \E_P\left[\tX^\alpha\right]
-\sum_{\varnothing \subseteq \beta \subsetneq \alpha} \tr_{\alpha\backslash\beta} \, \E_Q\left[\tX^{\beta}\right] = a^{|\alpha|}\E_P\left[X^\alpha\right]
-\sum_{\varnothing \subseteq \beta \subsetneq \alpha} \tr_{\alpha\backslash\beta} \, a^{|\beta|} \E_Q\left[X^{\beta}\right]. \]
By the inductive hypothesis, $\tr_{\alpha\backslash\beta} = a^{|\alpha \backslash \beta|} \, r_{\alpha\backslash\beta}$ and so by the equation above we are done, as $a^{|\beta|} a^{|\alpha \backslash \beta|} = a^{|\alpha|}$.
\end{proof}

\begin{remark}\label{rem:non_recursive} 
We show that the non-recursive formula for the $r_\alpha$ given in~\eqref{eq:r_non_rec} is equivalent to the recursive one in~\eqref{eq:r_def_graphs}. For convenience we restate~\eqref{eq:r_non_rec} below:
\begin{equation}\label{eq:r_non_rec_restate} r_\alpha(P, Q) = \sum_{\varnothing \subseteq \delta \subseteq \alpha} \E_{P}[X^\delta] \sum_{\tau \in \mathcal{P}(\alpha\backslash \delta)} (-1)^{|\tau|} \;|\tau|! \prod_{\gamma \in \tau}\E_Q[X^{\gamma}].
\end{equation}
This follows by induction. Small cases may be checked. For the inductive step, assume the formula~\eqref{eq:r_non_rec_restate} holds for all sets of size at most $k$ and let $\alpha$ be any set of size $k+1$. Then by~\eqref{eq:r_def_graphs}, and noting the induction assumption holds for each~$\beta$ we have the following, writing $a_\delta$ for $\E_P[X^\delta]$ and $b_\gamma$ for $\E_Q[X^\gamma]$,

\[r_\alpha(P,Q) = a_{\alpha} - \sum_{\varnothing \subseteq \beta \subsetneq \alpha } r_\beta b_{\alpha \backslash \beta} = a_\alpha - \sum_{\varnothing \subseteq \beta \subsetneq \alpha }  b_{\alpha \backslash \beta} \sum_{\varnothing\subseteq \delta \subseteq \beta} a_{\delta} \sum_{\tau \in \mathcal{P}(\beta \backslash \delta) } (-1)^{|\tau|} \; |\tau|! \prod_{\gamma \in \tau} b_\gamma.\]
We may now swap the order of summation to yield

\[r_\alpha(P,Q) = a_\alpha -\sum_{\varnothing\subseteq \delta \subsetneq \alpha} a_{\delta} \sum_{\varnothing \subseteq \beta \subseteq \alpha\backslash \delta }  b_{\alpha \backslash \beta}  \sum_{\tau \in \mathcal{P}(\beta \backslash \delta) } (-1)^{|\tau|} \; |\tau|! \prod_{\gamma \in \tau} b_\gamma.\]
For $\tau$ a partition of $\beta \backslash \delta$, let $\tau'$ be the partition constructed from $\tau$ by adding the part $\alpha \backslash \beta$ and note that $\tau'$ is a partition of $\alpha \backslash \delta$. Note each $\tau'$ appears $|\tau'|=|\tau|+1$ times in the sum above, and hence 

\[r_\alpha(P,Q) = a_\alpha -\sum_{\varnothing\subseteq \delta \subsetneq \alpha} a_{\delta}   \sum_{\tau' \in \mathcal{P}(\alpha \backslash \delta) } (-1)^{|\tau'|-1} \; |\tau'|! \prod_{\gamma \in \tau'} b_\gamma\]
which matches the non-recursive expression for $r$ as claimed.
\end{remark}

\needspace{3\baselineskip}
\section{Proof of Theorems~\ref{thm:Gaussian} and~\ref{thm:binary}}\label{sec:proof} 

In this section we give the full proofs of the main results, Theorems~\ref{thm:Gaussian} and~\ref{thm:binary}.

\begin{proof}[Proof of Theorem~\ref{thm:Gaussian}]

\emph{Hard regime.} We start by proving the computational lower bound.
By definition of the Additive Gaussian model, we can write the observed edge weights $Y=\{Y_{ij}\}_{i\leq j}$ as $Y=X+Z$, where $Z$ consists of i.i.d.\ $\mathcal{N}(0,1)$ entries, and
\begin{equation}
\label{eq:X_def}
X_{ij}=
\begin{cases}
\lambda/x_\ell & \qquad  \sigma_i=\sigma_j=\ell\;\;  \text{for some }\ell\in [M],\\
0 & \qquad  \text{otherwise}.
\end{cases}
\end{equation}
Recall the sequence $r_\alpha$, defined recursively via
\[
r_\alpha = \EE_\mathbb{P}\left[X^\alpha\right] - \sum_{0 \leq \beta\lneq \alpha}r_\beta \, \binom{\alpha}{\beta} \, \EE_\mathbb{Q}\left[X^{\alpha-\beta}\right].
\]
By Lemma~\ref{lem:adv-sep} and Proposition~\ref{prop:adv-gauss}, we have that if
\begin{equation}\label{eq:r.sum.bound}
\sum_{\alpha \,:\, |\alpha|\leq D}\frac{r_\alpha^2}{\alpha!}=1+o(1),
\end{equation}
then no degree-$D$ test can weakly separate $\mathbb{P}$ and $\mathbb{Q}$. Thus, to prove the computational lower bound, it suffices to show that~\eqref{eq:r.sum.bound} holds for $D^5\lambda^2M^2(k^2/n\vee 1)=o(1)$.

We will demonstrate~\eqref{eq:r.sum.bound} by proving the following three facts. (We consider the sets $\alpha$ as graphs and write $V(\alpha)$ for the vertex set and $C(\alpha)$ for the set of connected components, see Section~\ref{subsec:alpha_as_graph} for details.)

\begin{enumerate}[(i)]
\item\label{fact:factorization} For all $\alpha$, the term $r_\alpha$ factorizes over the connected components of $\alpha$. That is,
\[
r_\alpha = \prod_{\beta\in \mathcal{C}(\alpha)} r_\beta.
\]
\item\label{fact:tree} If at least one connected component of $\alpha$ is a tree, then $r_\alpha=0$.
\item\label{fact:r.bound} For all $\alpha$, where $|\alpha| = |E(\alpha)|$ counts the edges in the graph $\alpha$,
\[
\left|r_\alpha\right| \leq \left(|\alpha|+1\right)^{|\alpha|}\left(\frac{\widehat{M}\lambda}{C}\right)^{|\alpha|}\left(\frac{k}{n}\right)^{|V(\alpha)|}.
\]
\end{enumerate}
Fact~(\ref{fact:factorization}) follows directly from Lemma~\ref{lem:multiply}. To see Fact~(\ref{fact:tree}), we note that by Lemma~\ref{lem:forest} it suffices to show that for any $\tau$ a tree, we have $\E_\mathbb{P}\left[X^\tau\right]=\E_\mathbb{Q}\left[X^\tau\right]$. But recall that for a tree, the number of edges is one less than the number of vertices, i.e.\ $|\tau|=|V(\tau)|-1$ and $\tau$ consists of one connected component so that $|C(\tau)|=1$. Thus, by Lemma~\ref{lem:integralgenx} we are done. Fact~(\ref{fact:r.bound}) follows from Lemma~\ref{lem:r.bound}. We will state and prove Lemmas~\ref{lem:integralgenx} and~\ref{lem:r.bound} at the end of this section.

Next, we argue that the three facts combined yield~\eqref{eq:r.sum.bound}. From fact (\ref{fact:r.bound}), we have for each $\alpha$,
\begin{align*}
r_\alpha^2 \leq & \left(|\alpha|+1\right)^{2|\alpha|}\left(\frac{\widehat{M}\lambda}{C}\right)^{2|\alpha|}\left(\frac{k}{n}\right)^{2|V(\alpha)|}\\
= & \left(|\alpha|+1\right)^{2|\alpha|} \left(\frac{\widehat{M}^2\lambda^2 k^2}{C^2n}\right)^{|E(\alpha)|}\left(\frac{k^2}{n}\right)^{|V(\alpha)| - |E(\alpha)|} n^{-|V(\alpha)|}.
\end{align*}
From (\ref{fact:tree}), we know that $r_\alpha$ is nonzero only when all connected components of $\alpha$ contain at least one cycle. Denote
\[
\mathcal{G}_{d,v}=\left\{\alpha: \left|E(\alpha)\right|=d; \, \left|V(\alpha)\right|=v; \text{ for all }\beta\in \mathcal{C}(\alpha), \, \beta \text{ is not a tree}\right\}.
\]
Note that for all $d,v$ such that $v>d$, we have that $\mathcal{G}_{d,v}=\varnothing$ because if all connected components of $\alpha$ contains at least one cycle, we must have $|\alpha|\geq |V(\alpha)|$. Thus for $k^2\geq n$, we have shown that for all $d,v$ with $\alpha\in \mathcal{G}_{d,v}$,
\[
r_\alpha^2 \leq \left(d+1\right)^{2d} \left(\frac{\widehat{M}^2\lambda^2 k^2}{C^2n}\right)^{d} n^{-v}.
\]
On the other hand, for $k^2<n$, we have
\[
r_\alpha^2 \leq \left(d+1\right)^{2d}\left(\frac{\widehat{M}\lambda}{C}\right)^{2d}\left(\frac{k}{n}\right)^{2v}\leq \left(d+1\right)^{2d} \left(\frac{\widehat{M}\lambda}{C}\right)^{2d}n^{-v}.
\]
Combined with the bound on $r_\alpha^2$ for $k^2\geq n$, we have shown that
\[
r_\alpha^2 \leq \left(d+1\right)^{2d} \left(\left(\frac{\widehat{M}\lambda}{C}\right)^2\left(\frac{k^2}{n}\vee 1\right)\right)^{d} n^{-v}.
\]
Next, we bound the size of $\mathcal{G}_{d,v}$ by counting the number of graphs with exactly $d$ edges and $v$ vertices:
\begin{equation}
\label{eq:comb.bound}
\left|\mathcal{G}_{d,v}\right| \leq \binom{n}{v}\binom{v}{2}^d \leq n^v v^{2d}.
\end{equation}
where the factor $\binom{n}{v}$ enumerates the possibilities for the vertex set in $\alpha$; the $\binom{v}{2}^d$ factor counts the allocation of the $d$ edges, allowing for edge multiplicity. Combining~\eqref{eq:w.bound} (see Lemma~\ref{lem:r.bound} below) and~\eqref{eq:comb.bound} yields
\begin{align*}
\sum_{\alpha \,:\, |\alpha|\leq D}\frac{r_\alpha^2}{\alpha!}
&\leq  r_0^2+\sum_{d=1}^D\sum_{v=1}^d\sum_{\alpha\in \mathcal{G}_{d,v}}\frac{r_\alpha^2}{\alpha!}\\
&\leq  1+\sum_{d=1}^D\sum_{v=1}^d \left|\mathcal{G}_{d,v}\right|\cdot (d+1)^{2d} \left(\left(\frac{\widehat{M}\lambda}{C}\right)^2\left(\frac{k^2}{n}\vee 1\right)\right)^{d} n^{-v}\\
&\leq  1+\sum_{d=1}^D\sum_{v=1}^d n^v v^{2d}(d+1)^{2d} \left(\left(\frac{\widehat{M}\lambda}{C}\right)^2\left(\frac{k^2}{n}\vee 1\right)\right)^{d} n^{-v}\\
&=  1+\sum_{d=1}^D \left((d+1)^2\left(\frac{\widehat{M}\lambda}{C}\right)^2\left(\frac{k^2}{n}\vee 1\right)\right)^{d} \, \sum_{v=1}^d  v^{2d}\\
&\leq   1+ D\sum_{d=1}^D \left((D+1)^2 D^{2}\left(\frac{\widehat{M}\lambda}{C}\right)^2\left(\frac{k^2}{n}\vee 1\right)\right)^{d} \\
&\leq  1+ D\sum_{d=1}^D \left(\left(2D^2\frac{\widehat{M}\lambda}{C}\right)^2\left(\frac{k^2}{n}\vee 1\right)\right)^{d} \\
&=  1+(1+o(1)) \, 4D^5\left(\frac{\widehat{M}\lambda}{C}\right)^2\left(\frac{k^2}{n}\vee 1\right)=1+o(1),
\end{align*}
where the last two equalities follow by the condition $D^5(\widehat{M}\lambda)^2(k^2/n\vee 1)=o(1)$, under which the summation over $d$ is a geometrically decreasing sequence, dominated by the first term.

\bigskip

\noindent\emph{Easy regime.} Next, we show that in the ``easy'' regime $\lambda^2\widetilde{M}^2(k^2/n\vee 1)=\omega(1)$ and $\widetilde{M}^2k=\omega(\widehat{M}^2)$, there is a low-degree test that strongly separates $\mathbb{P}$ and $\mathbb{Q}$. When $\lambda^2\widehat{M}^2k^2/n=\omega(1)$, consider the algorithm that uses $\widehat{T}=\sum_{i}Y_{ii}$, the sum of the diagonal elements, as the test statistic. We can compute the first and second moments of $\widehat{T}$ under the two models using \eqref{eq:X_def} to note that under $\mathbb{P}$, we have $Y_{ii} = \frac{\lambda}{x_\ell} + \mathcal{N}(0,1)$ if $ \sigma_i=\sigma_j=\ell$ for some $\ell\in [M]$ and $Y_{ii} = \mathcal{N}(0,1)$ otherwise, where each community label $\ell$ is selected with probability $\frac{x_{\ell} k}{n}$ and no label is selected with probability $1-\frac{k}{n}$. Under  $\mathbb{Q}$, we replace $x$ and $M$ with $x'$ and $M'$.
\begin{align*}
\mathbb{E}_\mathbb{P} \left[\widehat{T}\right] &=  n \, \mathbb{E}_\mathbb{P} \left[Y_{11}\right] = n\left[\sum_{\ell \in [M]}\frac{\lambda}{x_\ell}\cdot\mathbb{P}\{\sigma_1\!=\!\ell\}+0\cdot \mathbb{P}\{\sigma_1\!=\!\star\}\right]  = n\!\!\sum_{\ell \in [M]} \frac{k\lambda}{n}=Mk\lambda,\\
\mathbb{E}_\mathbb{Q} \left[\widehat{T}\right] &=  n \, \mathbb{E}_\mathbb{Q} \left[Y_{11}\right] = M'k\lambda,\\
\Var_\PP\left[\widehat{T}\right] &\leq  n \, \EE_\PP\left[Y_{11}^2\right] = n\left[\sum_{\ell \in [M]}\left(\frac{\lambda^2}{x_\ell^2}+1\right) \cdot  \mathbb{P}\{\sigma_1\!=\!\ell\} + 1 \cdot \mathbb{P}\{\sigma_1\!=\!\star\}\right] = n+\!\!\sum_{\ell\in [M]}\frac{k\lambda^2}{x_\ell},\\
\Var_\QQ\left[\widehat{T}\right] &\leq  n \, \EE_\QQ\left[Y_{11}^2\right] =  n+\!\!\sum_{\ell\in [M']}\frac{k\lambda^2}{x'_\ell}.
\end{align*}
Note also that $M \min_\ell x_\ell \ge C$ implies $\max_\ell 1/x_\ell < M/C$; thus, $\sum_{\ell\leq M}\frac{1}{x_\ell}\leq M^2/C$. Hence, when~$\widetilde{M}^2k/\widehat{M}^2  = \omega(1)$ and $\widetilde{M}^2\lambda^2k^2/n=\omega(1)$,
\[
\sqrt{\max\left\{\Var_\QQ\left[\widehat{T}\right], \Var_\PP\left[\widehat{T}\right]\right\}} = o\left(\left|\EE_\PP\left[\widehat{T}\right] - \EE_\QQ \left[\widehat{T}\right]\right|\right).
\]
Thus, thresholding $\widehat{T}$ strongly separates $\PP$ and $\QQ$.
\end{proof}

\begin{proof}[Proof of Theorem~\ref{thm:binary}] \emph{Hard regime.} The proof proceeds by comparison to a corresponding Gaussian model, so that we can reuse the calculations in the proof of Theorem~\ref{thm:Gaussian}. Our starting point is Proposition~\ref{prop:adv-binary}. Define $X = X^{(q,s)}$ appropriately for our binary testing problem, i.e., $X_{ij}^{(q,s)} = q + s/x_\ell$ if $\sigma_i = \sigma_j = \ell$, and $X_{ij}^{(q,s)} = q$ otherwise. Let $\tau_0 = q$, and recall that we have a valid constant $\tau_1 < 1$ by assumption. Consider the additive Gaussian testing problem (as in Theorem~\ref{thm:Gaussian}) with the same parameters $n,k,M,x$ as our binary model, and with $\lambda := s/\sqrt{q(1-\tau_1)}$. Let $X^{(\lambda)}$ denote the corresponding $X$ as per Proposition~\ref{prop:adv-gauss}, i.e., $X_{ij}^{(\lambda)} = \lambda/x_\ell$ if $\sigma_i = \sigma_j = \ell$, and $X_{ij}^{(\lambda)} = 0$ otherwise. Note $X_{ij}^{(q,s)} = (s/\lambda)X_{ij}^{(\lambda)} + q$ and so by Lemmas~\ref{lem:constantshift} and~\ref{lem:scalar} we have, $r_\alpha(X^{(q,s)}) = (s/\lambda)^{|\alpha|} r_\alpha(X^{(\lambda)})$. By Proposition~\ref{prop:adv-binary},
\begin{align*}
\Adv_{\le D}(\PP,\QQ) &\le \sqrt{\sum_{\alpha \in \{0,1\}^N, \, |\alpha| \le D} \frac{r_\alpha\left(X^{(q,s)}\right)^2}{\left(q(1-\tau_1)\right)^{|\alpha|}}}
= \sqrt{\sum_{\alpha \in \{0,1\}^N, \, |\alpha| \le D} r_\alpha\left(X^{(\lambda)}\right)^2}\\
&\le \sqrt{\sum_{\alpha \in \NN^N, \, |\alpha| \le D} \frac{r_\alpha\left(X^{(\lambda)}\right)^2}{\alpha!}}.
\end{align*}
In other words, we have related the conclusion of Proposition~\ref{prop:adv-binary} to the conclusion of Proposition~\ref{prop:adv-gauss} but with $s/\sqrt{q(1-\tau_1)}$ in place of $\lambda$. The result now follows by the proof of Theorem~\ref{thm:Gaussian}.

\bigskip

\noindent\emph{Easy regime.}
We now consider a signed triangle count $\widehat{R}$ as our test statistic. Let 
\begin{equation}\label{eq:signedtridefn} \widehat{R}= \sum_{i < j < k} \sY_{ij}\sY_{ik}\sY_{jk} \qquad \:\:\:\:\mbox{where } \sY_{ij} = Y_{ij} - q.\end{equation}
Expectation and variance calculations for $\widehat{R}$ are computed in Lemma~\ref{lem:signtricalcs} of Section~\ref{subsec:trianglecounts}. Denote by $\widehat{M}$ the maximum of~$M$ and $M'$. Then,
\[
\left|\mathbb{E}_\mathbb{P}\left[\widehat{R}\right]-\mathbb{E}_\mathbb{Q}\left[\widehat{R}\right]\right| \:  = \: \: \tfrac{1}{3}\left|M-M'\right| \, s^3 \, k^3 \left(1+O(n^{-1})\right),\]
and
\begin{align}\notag &\!\!\!\max\left\{\Var_\PP  \left[\widehat{R}\right],  \Var_\QQ\left[\widehat{R}\right]\right\} \: \\ 
\label{eq:varTri} & \leq  \frac{1}{C}\widehat{M}^2 k^5 s^6  + \widehat{M} k^4s^4 q +   \frac{1}{C}\widehat{M}^2k^4s^5  + \tfrac{1}{3}n^3q^3 + n k^2 sq^2  + k^3q^2s + k^3qs^2 + \tfrac{1}{3}\widehat{M}k^3s^3
\end{align}
where $C$ is the constant from the assumption that $M \min_\ell x_\ell, M' \min_\ell x'_\ell > C$. Writing $\widetilde{M}=|M-M'|$, notice that to prove strong separation it suffices to show that each term in \eqref{eq:varTri} is $o(\widetilde{M}^2 s^6 k^6)$. For the fourth term to be $o(\widetilde{M}^2 s^6 k^6)$ is equivalent to $\widetilde{M}^{2/3}s^2k^2/(nq)=\omega(1)$, one of our assumptions. Similarly, for the first term to be $o(\widetilde{M}^2 s^6 k^6)$ is equivalent to $\widetilde{M}^{2}k\widehat{M}^{2}=\omega(1)$ another of our assumptions.
For the last term to be $o(\widetilde{M}^2 s^6 k^6)$ is equivalent to $\widetilde{M}^{2/3}\widehat{M}^{-1/3}sk=\omega(1)$, which is implied by our assumption $\widehat{M}^{-1/3}sk=\omega(1)$. All other terms follow also due to these assumptions.
\end{proof}

\begin{lemma}\label{lem:integralgenx}
For each $\alpha \in \NN^N$, and for each $x=(x_1, \ldots, x_c)$ with $\sum_\ell x_\ell=1$,
\begin{equation*}
\mathbb{E}_\mathbb{P} \left[X^\alpha\right] = \lambda^{|\alpha|}\left(\frac{k}{n}\right)^{|V(\alpha)|}\prod_{\beta \in C(\alpha)} \, \sum_{\ell=1}^c x_\ell^{|V(\beta)|-|\beta|}.
\end{equation*}
\end{lemma}

\begin{proof} First consider $\beta$ a connected graph. Note that for $(i,j)\in \beta$, if it is not the case that $i,j \in S_\ell$ for some $\ell$ then $X^{(i,j)}\sim \mathcal{N}(0,1)$ and so $\mathbb{E}_\mathbb{P}\left[X^{(i,j)}\right]=0$ (here, we have used that our $S_\ell$'s do not overlap). Hence, for $\beta$ connected,
\[
\mathbb{E}_{\mathbb{P}} \left[X^\beta\right] =   \sum_{\ell=1}^c \mathbb{P} \left( V(\beta) \in S_\ell \right) \left( \frac{\lambda}{x_\ell} \right)^{|\beta|}   = \sum_{\ell=1}^c   \left(\frac{x_\ell \, k}{n}\right)^{|V(\beta)|}  \left( \frac{\lambda}{x_\ell} \right)^{|\beta|}.
\]
Notice, it is now enough to show that the $X^\beta$'s are independent for $\beta$'s connected components of $\alpha$, as this would imply that $\mathbb{E}_{\mathbb{P}} \left[X^\alpha\right] = \prod_{\beta \in C(\alpha)} \mathbb{E}_{\mathbb{P}} \left[X^\beta\right]$ and we have the result.

This independence follows because $X^{(i,j)}$ depends only on the events $[i \in S_\ell]$, $[j \in S_{\ell'}]$ for each $\ell, \ell'$; thus, $X^\beta$ and $X^{\beta'}$ are independent as long as their vertex sets $V(\beta)$ and $V(\beta')$ do not overlap.  As the vertex sets of connected components are mutually non-overlapping, we have finished the proof.
\end{proof}

\begin{lemma}\label{lem:r.bound}
Suppose $Mx_{(1)}\geq C$ and $M'x'_{(1)} \geq C$ where $x_{(1)} := \min_\ell x_\ell$ for some constant $C>0$. Then there exists 
$n_0 \in \mathbb{N}$ such that for $n>n_0$, for all $\alpha$, we have that $r_\alpha$ satisfies
\begin{equation}\label{eq:w.bound}
|r_\alpha|\leq \left(|\alpha+1|\right)^{|\alpha|}\left(\frac{\widehat{M}\lambda}{C}\right)^{|\alpha|}\left(\frac{k}{n}\right)^{|V(\alpha)|}.
\end{equation}
\end{lemma}

\begin{proof}
We will argue by induction on $|\alpha|$. A graph $\alpha$ with $|\alpha|=1$ is either a tree with two vertices and one edge, or a self-loop with one vertex and one edge. If $\alpha$ is a tree, then $r_\alpha=0$ and~\eqref{eq:w.bound} trivially holds. Recall $\widehat{M}=\max\{M, M'\}$. If $\alpha$ is a self-loop, we have by Lemma~\ref{lem:integralgenx},
\[ r_\alpha = \mathbb{E}_\mathbb{P}\left[X^\alpha\right] - \mathbb{E}_\mathbb{Q}\left[X^\alpha\right] =  M\lambda \left(\frac{k}{n}\right) - M'\lambda \left(\frac{k}{n}\right) \leq  \left(|\alpha|+1\right)^{|\alpha|}\left(\frac{\widehat{M}\lambda}{C}\right)^{|\alpha|}\left(\frac{k}{n}\right)^{|V(\alpha)|}, \]
where the last inequality is because $C\leq M x_{(1)}\leq 1$. We have shown~\eqref{eq:w.bound} for $|\alpha|=1$. Suppose~\eqref{eq:w.bound} holds for all $\alpha$ with $|\alpha|\leq d-1$; next, we show it also holds for $|\alpha|=d$. 

If $\alpha$ is not connected, then each connected component $\beta\in \mathcal{C}(\alpha)$ has $|\beta|<d$. Thus, from the factorization lemma and the induction hypothesis, we have
\[
\left|r_\alpha\right|= \prod_{\beta\in \mathcal{C}(\alpha)}\left|r_\beta\right| \leq \prod_{\beta\in \mathcal{C}(\alpha)}\left(|\beta|+1\right)^{|\beta|}\left(\frac{\lambda \widehat{M}}{C}\right)^{|\beta|}\left(\frac{k}{n}\right)^{|V(\beta)|}\leq \left(|\alpha|+1\right)^{|\alpha|}\left(\frac{\lambda \widehat{M}}{C}\right)^{|\alpha|}\left(\frac{k}{n}\right)^{|V(\alpha)|}.
\]
Thus~\eqref{eq:w.bound} holds. Next, we show~\eqref{eq:w.bound} for $\alpha$ connected. If $\alpha$ is a tree, then by Fact~(\ref{fact:tree}) we have $r_\alpha=0$ and~\eqref{eq:w.bound} holds. Therefore, it suffices to consider $\alpha$ that is not a tree. Recall that
\begin{equation}\label{eq:w.expand}
r_\alpha = \mathbb{E}_\mathbb{P}\left[X^\alpha\right] - \mathbb{E}_\mathbb{Q}\left[X^\alpha\right]
- \sum_{0<\beta\lneq\alpha}r_\beta\binom{\alpha}{\beta} \mathbb{E}_\mathbb{Q}\left[X^{\alpha\backslash\beta}\right].
\end{equation}
For the first term in~\eqref{eq:w.expand}, we can apply Lemma~\ref{lem:integralgenx} for connected $\alpha$:
\begin{align*}
\mathbb{E}_\mathbb{P}\left[X^\alpha\right] &=  \lambda^{|\alpha|}\left(\frac{k}{n}\right)^{|V(\alpha)|}\sum_{\ell \in [M]} x_\ell^{|V(\alpha)|-|\alpha|}\\
& \leq  \lambda^{|\alpha|}\left(\frac{k}{n}\right)^{|V(\alpha)|} Mx_{(1)}^{|V(\alpha)|-|\alpha|}\\
& = \lambda^{|\alpha|} \left(\frac{k}{n}\right)^{|V(\alpha)|} \left(Mx_{(1)}\right)^{|V(\alpha)| - |\alpha|} M^{1-|V(\alpha)|+|\alpha|}\\
& \leq  \left(\frac{\lambda M}{C}\right)^{|\alpha|}\left(\frac{k}{n}\right)^{|V(\alpha)|} M^{1-|V(\alpha)|}\\
& \leq \left(\frac{\lambda M}{C}\right)^{|\alpha|}\left(\frac{k}{n}\right)^{|V(\alpha)|} 
\end{align*}
for large enough $n$. The first inequality is because assuming $\alpha$ is not a tree, $|V(\alpha)|\leq |\alpha|$ and $x_\ell\leq 1$; the second inequality is because $1\geq Mx_{(1)}\geq C$, thus $(Mx_{(1)})^{|V(\alpha)|-|\alpha|}\leq C^{|V(\alpha)|-|\alpha|}\leq C^{-|\alpha|}$; the last inequality holds since $|V(\alpha)|\geq 1$.

Next, we bound the third term in~\eqref{eq:w.expand}. For each $\beta\lneq \alpha$ that is nonempty, $|\beta|<|\alpha|=d$. From the induction hypothesis, we have
\[
\left|r_\beta\right| \leq \left(|\beta|+1\right)^{|\beta|}\left(\frac{\widehat{M}\lambda}{C}\right)^{|\beta|}\left(\frac{k}{n}\right)^{|V(\beta)|}.
\]
Thus,
\begin{align*}
    & \left|r_\beta \mathbb{E}_{\mathbb{Q}}X^{\alpha\backslash\beta}\right|
    \leq \left(|\beta|+1\right)^{|\beta|}\left(\frac{\widehat{M}\lambda}{C}\right)^{|\beta|}\left(\frac{k}{n}\right)^{|V(\beta)|}\cdot \lambda^{|\alpha\backslash\beta|}\left(\frac{k}{n}\right)^{|V(\alpha\backslash\beta)|}\prod_{\gamma\in \mathcal{C}(\alpha\backslash\beta)}\,\sum_{\ell \in [M']}(x'_\ell)^{|V(\gamma)|-|\gamma|}\\
    & = \left(|\beta|+1\right)^{|\beta|}\left(\frac{\widehat{M}\lambda}{C}\right)^{|\alpha|}\left(\frac{k}{n}\right)^{|V(\alpha)|} \left(\frac{\widehat{M}}{C}\right)^{-|\alpha\backslash\beta|}\left(\frac{k}{n}\right)^{-|V(\beta)\cap V(\alpha\backslash\beta)|} \prod_{\gamma\in \mathcal{C}(\alpha\backslash\beta)}\,\sum_{\ell\in [M']}(x'_\ell)^{|V(\gamma)|-|\gamma|}\\
    &\leq 
     \left(|\beta|+1\right)^{|\beta|}\left(\frac{\widehat{M}\lambda}{C}\right)^{|\alpha|}\left(\frac{k}{n}\right)^{|V(\alpha)|} \prod_{\gamma\in \mathcal{C}(\alpha\backslash\beta)}\left(\frac{\widehat{M}}{C}\right)^{-|\gamma|}\sum_{\ell\in [M']}(x'_\ell)^{|V(\gamma)|-|\gamma|}.
\end{align*}
Next, we show that for all $\gamma\in \mathcal{C}(\alpha\backslash\beta)$, we have that $(\widehat{M}/C)^{-|\gamma|}\sum_{\ell\in [M']}(x'_\ell)^{|V(\gamma)|-|\gamma|}\leq 1$.
Note that $|V(\gamma)|\leq |\gamma|+1$. We discuss the cases $|V(\gamma)|= |\gamma|+1$ and $|V(\gamma)|\leq |\gamma|$ separately. If $|V(\gamma)|= |\gamma|+1$, then
\[
\left(\frac{\widehat{M}}{C}\right)^{-|\gamma|}\sum_{\ell\in [M']}(x'_\ell)^{|V(\gamma)|-|\gamma|}= \left(\frac{\widehat{M}}{C}\right)^{-|\gamma|}\leq 1.
\]
If $|V(\gamma)|\leq |\gamma|$, then we have
\begin{align*}
\left(\frac{\widehat{M}}{C}\right)^{-|\gamma|}\sum_{\ell \in [M']} (x'_\ell)^{|V(\gamma)|-|\gamma|} \leq  \left(\frac{M}{C}\right)^{-|\gamma|} \, M' \,  (x'_{(1)})^{|V(\gamma)|-|\gamma|} &\stackrel{(a)}{\leq}  \left(\widehat{M}x'_{(1)}\right)^{|V(\gamma)|-|\gamma|} \, \widehat{M}^{1-|V(\gamma)|} \, C^{|\gamma|}\\
&\stackrel{(b)}{\leq} C^{|V(\gamma)|} \, \widehat{M}^{1-|V(|\gamma|)|}\stackrel{(c)}{\leq} 1,
\end{align*}
where (a) is from $M'\leq M$; (b) is from $M'x'_{(1)}\geq C$ and $|V(\gamma)|\leq |\gamma|$; (c) is from $C\leq 1$, $\widehat{M}\geq 1$, and $|V(\gamma)|\geq 1$ for all $\gamma\in \mathcal{C}(\alpha\backslash \beta)$.
We have shown that
\[
\left|r_\beta \mathbb{E}_{\mathbb{Q}}X^{\alpha\backslash\beta}\right|
    \leq \left(|\beta|+1\right)^{|\beta|}\left(\frac{\widehat{M}\lambda}{C}\right)^{|\alpha|}\left(\frac{k}{n}\right)^{|V(\alpha)|}.
\]
Plug in the values of $\mathbb{E}_\mathbb{P}[X^\alpha]$ and $\mathbb{E}_\mathbb{Q}[X^\alpha]$ to~\eqref{eq:w.expand} to obtain
\begin{align*}
\left|r_\alpha\right| &\leq  \left(\frac{\widehat{M}\lambda}{C}\right)^{|\alpha|}\left(\frac{k}{n}\right)^{|V(\alpha)|} + \lambda^{|\alpha|}\left(\frac{k}{n}\right)^{|V(\alpha)|}+\sum_{0<\beta\lneq\alpha}\binom{\alpha}{\beta} \left(|\beta|+1\right)^{|\beta|}\left(\frac{\widehat{M}\lambda}{C}\right)^{|\alpha|}\left(\frac{k}{n}\right)^{|V(\alpha)|}\\
& \leq \left[1+1+\sum_{0<\beta\lneq\alpha}\left(|\beta|+1\right)^{|\beta|}\right]\left(\frac{\widehat{M}\lambda}{C}\right)^{|\alpha|}\left(\frac{k}{n}\right)^{|V(\alpha)|}\\
&\leq  \left(|\alpha|+1\right)^{|\alpha|}\left(\frac{\widehat{M}\lambda}{C}\right)^{|\alpha|}\left(\frac{k}{n}\right)^{|V(\alpha)|},
\end{align*}
where the last inequality is because
\begin{align*}
 2+\sum_{0<\beta\lneq\alpha}\binom{\alpha}{\beta} \left(|\beta|+1\right)^{|\beta|} &= 2+\sum_{0<\ell<|\alpha|}\binom{|\alpha|}{\ell}(\ell+1)^\ell \leq  \left(|\alpha|+1\right)^{|\alpha|}.
\end{align*}
We have shown that~\eqref{eq:w.bound} holds for all $\alpha$.
\end{proof}

\needspace{3\baselineskip}
\section{Additional proofs}
\label{sec:additional-proofs}

\subsection{Proof of Lemma~\ref{lem:adv-sep}}

First we prove the statement for weak separation. Assume, for the sake of contradiction, that some degree-$D$ test $g: \RR^N \to \RR$ weakly separates $\PP$ and $\QQ$. Without loss of generality, we can shift and scale $g$ so that $\EE_\QQ[g] = 0$ and $\EE_\PP[g] = 1$. Weak separation guarantees that for sufficiently large $n$, $\Var_\QQ[g] = \EE_\QQ[g^2] \le C$ for some positive constant $C > 0$. Defining $f = g + C$, we have
\[ \Adv_{\le D} \ge \frac{\EE_\PP[f]}{\sqrt{\EE_\QQ[f^2]}} = \frac{1+C}{\sqrt{\EE_\QQ[g^2]+C^2}} \ge \frac{1+C}{\sqrt{C+C^2}} = \sqrt{\frac{1+C}{C}}, \]
which is a constant strictly greater than $1$, contradicting $\Adv_{\le D} = 1+o(1)$. The proof for strong separation is identical, except now $C = o(1)$.

\subsection{Proof of Proposition~\ref{prop:adv-gauss}}

The proof is similar to the proof of Theorem~2.2 in~\cite{SW-estimation}, so we only explain the differences. Our distribution $\QQ$ plays the role of the single ``planted'' distribution in~\cite{SW-estimation}. The only difference is that the quantity $\EE[f(Y)x]$ from~\cite{SW-estimation} needs to be replaced by our $\EE_\PP[f(Y)]$, which means (in the notation of~\cite{SW-estimation}) the vector $c$ needs to be redefined as $c_\alpha = \EE_\PP[h_\alpha(Y)] = \EE_\PP[X^\alpha]/\sqrt{\alpha!}$.

\subsection{Proof of Proposition~\ref{prop:adv-binary}}

Follow the proof of Theorem~2.7 in~\cite{SW-estimation}, but redefine $c = (c_\alpha)_{\alpha \in \{0,1\}^N}$ by $c_\alpha = \EE_\PP[\widetilde{X}^\alpha]$ where $\widetilde{X}_i = (\mu + 1/\mu)X_i - 1/\mu$ and $\mu = \sqrt{\frac{1-\tau_1}{\tau_0}}$. This gives the bound
\[ \Adv_{\le D}(\PP,\QQ)
\leq \sqrt{\sum_{\alpha \in \{0,1\}^{N} , \, |\alpha|\leq D}\frac{r_\alpha(\widetilde X)^2}{(1+\tau_0-\tau_1)^{2|\alpha|}}} \]
where $r_\alpha(\widetilde X)$ is defined in~\eqref{eq:r-def}. Using Lemmas~\ref{lem:constantshift} and~\ref{lem:scalar}, we have $r_\alpha(\widetilde X) = (\mu+1/\mu)^{|\alpha|} r_\alpha(X)$, so the above simplifies to give the result.

\subsection{Proof of Lemma~\ref{lem:constantshift}}

\emph{Base case(s).}
Note that by definition $\tr_\varnothing = r_\varnothing = 1$. Let $|\alpha|=1$, i.e.\ $\alpha=\{ij\}$ for some $1 \leq i \leq j \leq n$. Then the base step follows directly from the definition
\[ 
\tr_\alpha= \E_P\left[\tX^{ij}\right]-\E_Q\left[\tX^{ij}\right] = \E_P\left[X^{ij}\right]+y_{ij}-\E_Q\left[X^{ij}\right]-y_{ij}=r_\alpha.
\] 

\noindent\emph{Inductive step.} Fix $\alpha$ with $|\alpha|>1$ and assume $\tr_\beta=r_\beta$ for all $\beta \subsetneq \alpha$. Directly from the definition of $r$ and the inductive hypothesis,
\begin{equation}
\begin{split}
\label{eq:tilder}
\tr_\alpha
&=\E_P\left[\tX^\alpha\right]-\E_Q\left[\tX^\alpha\right] - \sum_{\varnothing \subsetneq \beta \subsetneq \alpha} \tr_{\alpha\backslash\beta} \,\E_Q\left[\tX^\beta\right] \\
&= \E_P\left[\tX^\alpha\right]-\E_Q\left[\tX^\alpha\right] - \sum_{\varnothing \subsetneq \beta \subsetneq \alpha} r_{\alpha\backslash\beta} \,\E_Q\left[\tX^\beta\right].
\end{split}
\end{equation}
We consider the third term, call it $\Asterisk$. Writing $y^{\eta}$ to indicate $\prod_{ij \in \eta} y_{ij}$, first notice that
\begin{equation}
\label{eq:tildeX}
\E_Q\left[\tX^\beta\right]
= \E_Q\left[\prod_{ij \in \beta}(X_{ij}+y_{ij})\right]
= \E_Q\left[X^\beta\right] + \sum_{\varnothing \subsetneq \eta \subseteq \beta} y^\eta \, \E_Q\left[X^{\beta\backslash \eta}\right];
 \end{equation}
hence,
\[
\Asterisk=\sum_{\varnothing \subsetneq \beta \subsetneq \alpha} r_{\alpha\backslash\beta} \, \E_Q\left[X^{\beta}\right] + \sum_{\varnothing \subsetneq \beta \subsetneq \alpha} r_{\alpha\backslash\beta}  \sum_{\varnothing \subsetneq \eta \subseteq \beta} y^\eta \, \E_Q\left[X^{\beta\backslash \eta}\right]. 
 \]
If we let $\beta'=\beta\backslash \eta$, instead of summing over $\varnothing  \subsetneq \beta \subsetneq \alpha$ and then $\varnothing \subsetneq \eta \subseteq \beta$, we may sum over $\varnothing \subsetneq \eta \subsetneq \alpha$ then $\varnothing \subseteq \beta' \subsetneq \alpha\backslash \eta$. Thus, noting also that $\alpha \backslash \beta = (\alpha\backslash \eta)\backslash \beta'$,
\[
\Asterisk=\sum_{\varnothing \subsetneq \beta \subsetneq \alpha} r_{\alpha\backslash\beta} \, \E_Q\left[X^{\beta}\right] + \sum_{\varnothing \subsetneq \eta \subsetneq \alpha } y^\eta \sum_{\varnothing \subseteq \beta' \subsetneq \alpha\backslash\eta} r_{(\alpha\backslash \eta)\backslash\beta'} \,  \E_Q\left[X^{\beta'}\right].
 \]
But, by the definition of $r_{\alpha\backslash \eta}$,
\begin{align*}
\sum_{\varnothing \subseteq \beta' \subsetneq \alpha\backslash\eta} r_{(\alpha\backslash \eta)\backslash\beta'}   \, \E_Q\left[X^{\beta'}\right] &= r_{\alpha\backslash\eta} + \sum_{\varnothing \subsetneq \beta' \subsetneq \alpha\backslash\eta} r_{(\alpha\backslash \eta)\backslash\beta'}  \,  \E_Q\left[X^{\beta'}\right] \\
&= \E_P\left[X^{\alpha\backslash \eta}\right] - \E_Q\left[X^{\alpha\backslash \eta}\right],
\end{align*}
which gives the following expression for $\Asterisk$ where we no longer have the sum over $\beta'$:
\[
\Asterisk=\sum_{\varnothing \subsetneq \beta \subsetneq \alpha} r_{\alpha\backslash\beta} \, \E_Q\left[X^{\beta}\right] + \sum_{\varnothing \subsetneq \eta \subsetneq \alpha } y^\eta \, \left(\E_P\left[X^{\alpha\backslash \eta}\right] - \E_Q\left[X^{\alpha\backslash \eta}\right]\right).
 \]
Substituting this expression for $\Asterisk$ into our original expression for $\tr_\alpha$ in \eqref{eq:tilder}, we have
\[
\tr_\alpha= \E_P\left[\tX^\alpha\right]-\E_Q\left[\tX^\alpha\right]
-\sum_{\varnothing \subsetneq \beta \subsetneq \alpha} r_{\alpha\backslash\beta} \, \E_Q\left[X^{\beta}\right] - \sum_{\varnothing \subsetneq \eta \subsetneq \alpha } y^\eta \, \left(\E_P\left[X^{\alpha\backslash \eta}\right] - \E_Q\left[X^{\alpha\backslash \eta}\right]\right). \]
However, using \eqref{eq:tildeX} (and a similar result on $\E_P\left[\tX^\alpha\right]$), we see that this last term is precisely what we need to cancel with the difference between $\E_P\left[\tX^\alpha\right]$ and $\E_P\left[X^\alpha\right]$ and the difference between $\E_Q\left[\tX^\alpha\right]$ and $\E_Q\left[X^\alpha\right]$, as
\[\sum_{\varnothing \subsetneq \eta \subsetneq \alpha } y^\eta \, \left(\E_P\left[X^{\alpha\backslash \eta}\right] - \E_Q\left[X^{\alpha\backslash \eta}\right]\right) = \sum_{\varnothing \subsetneq \eta \subseteq \alpha } y^\eta \, \left(\E_P\left[X^{\alpha\backslash \eta}\right] - \E_Q\left[X^{\alpha\backslash \eta}\right]\right).\]
Therefore,
\[
\tr_\alpha= \E_P\left[X^\alpha\right]-\E_Q\left[X^\alpha\right]
-\sum_{\varnothing \subsetneq \beta \subsetneq \alpha} r_{\alpha\backslash\beta} \, \E_Q\left[X^{\beta}\right] = r_\alpha, \]
and we have proven the inductive step.

\subsection{Calculations for signed triangle counts}\label{subsec:trianglecounts}
In this section we analyse the degree 3 signed triangle count test statistic $\widehat{R}$, defined in~\eqref{eq:signedtridefn}, and show bounds on the expectation and variance of $\widehat{R}$, which will prove it strongly separates $\PP$ and $\QQ$ in the easy regime. Recall,
\[ \widehat{R}= \sum_{i < j < k} \sY_{ij} \, \sY_{ik} \, \sY_{jk} \qquad \:\:\:\:\mbox{where } \sY_{ij} = Y_{ij} - q.\]
\begin{lemma}\label{lem:signtricalcs} We let $\mathbb{P}=\mathbb{P}_{\rm Binary}(n,k,q,s,M,x)$, given parameters $n,k,q,s,M$ and $x \in \RR^M$ with $\sum_{\ell \in [M]} x_\ell=1$. Assume that $M \min_\ell x_\ell \ge C$. Then,
\begin{align*}
\mathbb{E}_\mathbb{P} \left[\widehat{R}\right] & = \frac{1}{3} Ms^3k^3\left(1+O(n^{-1})\right),\\
\Var_\PP\left[\widehat{R}\right] & \leq   \frac{1}{C}M^2k^5 s^6 + M k^4s^4 q +   \frac{1}{C}M^2k^4s^5 + \frac{1}{3}n^3q^3 + n k^2 sq^2 + k^3q^2s + k^3qs^2 + \frac{1}{3}Mk^3s^3.
\end{align*}
\end{lemma}
\begin{proof}
Recall that in our model, the binary random variable $Y_{ij}$ takes value $1$ with probability $q+s/x_c$ if $\sigma_i=\sigma_j=c$ for some $c\in [M]$ and takes value 1 with probability $q$ otherwise.
Thus, we may calculate the expected values of $\sY_{ij}$ conditioned on the community assignments of $i$ and~$j$:
\begin{eqnarray}
\label{eq:expsY} 
\E_{\PP}\left[ \sY_{ij} \: | \: \sigma_i=c_i, \, \sigma_j=c_j \right]
& = &   \begin{cases}
            \frac{s}{x_c}    & \mbox{ if $c_i=c_j=c$ for some $c \in [M]$,}\\
            0   & \mbox{ otherwise.}
        \end{cases}
\end{eqnarray}
We now split the proof into expectation and variance calculations. All probabilities, expectations and variances will be with respect to $\PP$, but we drop the subscript.

\bigskip

\noindent \textbf{\emph{Expectation.}} Let \[\Ntri = \left\{ \{ij, \, ik, \, jk \} \: : \: i, j, k \in [n], \: i < j < k \right\},\]
and then we may express the signed triangle count $\widehat{R}$ by $\widehat{R} = \sum_{S \in \Ntri} \sY_S$. Fix a set of edges in $\Ntri$, w.l.o.g.\ $S= \{12, 13, 23 \}$. Then, writing $[M]_\star$ for the set $\{\star, 1, \ldots, M\}$ (recall $\star$ denotes no community membership),
\begin{align*} 
\E\left[\sY_{S}\right] 
    & =  \sum_{c_1,c_2, c_3 \in [M]_\star\:} \E\left[\sY_{12} \,  \sY_{13} \, \sY_{23} \; | \sigma_1=c_1, \, \ldots, \, \sigma_3=c_3 \right] P\left(\sigma_1=c_1, \, \ldots, \, \sigma_3=c_3 \right)  \\
    & =  \sum_{c_1,c_2, c_3 \in [M]_\star\:}\: \prod_{ij \in \{12, 13, 23 \} } \E\left[\sY_{ij} \: | \: \sigma_1=c_1, \,  \ldots, \, \sigma_3=c_3 \right] \:\prod_{i=1}^3 P\left(\sigma_i=c_i\right), 
\end{align*}
as the expected values of $\sY_{ij}$ and $\sY_{ik}$ are independent conditional on the community assignments of $i,j,k$. Note by \eqref{eq:expsY}, $\E[\sY_{ij} | \sigma_i=c_i, \sigma_j=c_j]$ is equal to zero unless $c_i=c_j=c$ for some $c \in [M]$. Therefore the only non-zero terms in the sum above are those for which $c_1=c_2=c_3=c$ for some $c \in [M]$. Let $\mono$ be the event that $\sigma_1=\sigma_2=\sigma_3=c$, then
\[
\E\left[\sY_{S}\right] 
     =  \sum_{c=1}^M \: \prod_{ij \in S } \! \E \left[\sY_{ij} \: | \: \mono \right] \: \prod_{i=1}^3 P\left(\sigma_i\!=\!c\right) 
     =  \sum_{c=1}^M \frac{s^3}{x_c^3} \left(\frac{k x_c}{n}\right)^3 = M k^3 s^3 n^{-3}.
\]
Because $|\Ntri|=\binom{n}{3}=\tfrac{1}{3}n^3(1+O(\frac{1}{n}))$, the expectation of $\widehat{R}$ is as claimed. 

\bigskip

\noindent \textbf{\emph{Variance.}}
Recall $\widehat{R} = \sum_{S \in \Ntri} \sY_S$ and so the variance is
\[ \Var\left[\widehat{R}\right] = \sum_{S, T \in \Ntri} \E\left[\sY_S \sY_T\right] - \E\left[\sY_S]\E[\sY_T\right].\]
Note that if $V(S) \cap V(T) = \varnothing$, i.e.\ the sets of pairs have no vertices in common, then $\sY_S$ and $\sY_T$ are independent and these terms cancel in the expression above. Hence we need only sum over $S, T$ with one overlapping vertex, with two overlapping vertices or equivalently one overlapping edge and lastly with all three vertices overlapping or equivalently $S=T$. Thus 
\begin{equation}\label{eq:countoverlaps123}
\Var\left[\widehat{R}\right] \leq \sum_{\substack{ S, T \in \Ntri\\ |V(S)\cap V(T)|=1}} \E\left[\sY_S \sY_T\right] + \sum_{\substack{ S, T \in \Ntri\\ |S \cap T|=1}} \E\left[\sY_S \sY_T\right]  + \sum_{ S \in \Ntri}  \E\left[\sY_S^2\right]. 
\end{equation}
The terms above correspond to the sets of pairs overlapping as $\gbowtiecol$, $\dkitecol$ and $\dddtriancol$ respectively where the gray edges denote pairs in $S$ and the pink edges denote pairs in $T$. 

We begin by bounding the first term in \eqref{eq:countoverlaps123}, i.e.\ that corresponding to $\gbowtiecol$. Fix some pair of sets which overlap on one vertex, w.l.o.g.\ $S_1= \{12, 13, 23 \}$ and $T_1=\{ 14, 15, 45\}$. Then, similarly to the expectation, again writing $[M]_\star$ for the set $\{\star, 1, \ldots, M\}$,
\begin{eqnarray*} 
\E[\sY_{S_1} \sY_{T_1}] 
    & = & \sum_{c_1,\ldots, c_5 \in [M]_\star\:}\: \prod_{ij \in S_1 \cup T_1
    } \E[\sY_{ij} \: | \: \sigma_1\!=\!c_1, \ldots \sigma_5\!=\!c_5 ] \:\prod_{i=1}^5 P(\sigma_i\!=\!c_i) 
\end{eqnarray*}
as the expected values of $\sY_{ij}$ and $\sY_{ik}$ are independent conditional on the community assignments of $i,j,k$. Note that $\E[\sY_{ij} | \sigma_i=c_i, \sigma_j=c_j]$ is equal to zero unless $c_i=c_j=c$ for some $c \in [M]$. Therefore the only non-zero terms in the sum above are those for which $c_1=\ldots=c_5=c$ for some $c \in [M]$. Let $\mono$ be the event that $\sigma_1 = \ldots = \sigma_5 = c$, then
%
\[
\E[\sY_{S_1} \sY_{T_1}] 
     =  \sum_{c=1}^M \: \prod_{ij \in S_1 \cup T_1 } \! \E[\sY_{ij} \: | \: \mono ] \: \prod_{i=1}^5 P(\sigma_i\!=\!c) \\
     =  \sum_{c=1}^M \frac{s^6}{x_c^6} \left(\frac{k x_c}{n}\right)^5 =  k^5 s^6 n^{-5} \sum_{c=1}^M \frac{1}{x_c}.
\]
Since there are at most $n^5$ ways we may pick $S,T\in \Ntri$ with $|V(S)\cap V(T)|=1$, we may conclude that the first term of \eqref{eq:countoverlaps123} is at most $k^5 s^6 \sum_{c=1}^M \frac{1}{x_c}$.

We next bound the second term in \eqref{eq:countoverlaps123}, i.e.\ that corresponding to $\dkitecol$. Similarly to above, fix some pair of sets which overlap on one edge, w.l.o.g.\ $S_2= \{12, 13, 23 \}$ and $T_2=\{ 12, 14, 24\}$. 
\begin{eqnarray} 
\notag \E[\sY_{S_2} \sY_{T_2}] 
    & = & \sum_{\underline{c} \in [M]_\star^4 } \E[\sY_{12}^2 \sY_{13} \sY_{23} \sY_{14} \sY_{24} \; | \; \sigma_1=c_1, \ldots, \sigma_4=c_4 ] \:\prod_{i=1}^4 P(\sigma_i=c_i)  \\
    & = & \sum_{\underline{c} \in [M]_\star^4 }\: \E[\sY_{12}^2 \: | \: \underline{\sigma}=\underline{c} ] \prod_{ij \in \{13, 23, 14, 24\} } \E[\sY_{ij} \: | \: \underline{\sigma}=\underline{c} ] \:\prod_{i=1}^4 P(\sigma_i=c_i) \label{eq:prodofEij}
\end{eqnarray}
since, as before, $\sY_{ij}$ and $\sY_{ik}$ are independent when we have conditioned on the community assignments of $i,j,k$.
Again, recall $\E[\sY_{ij} | \sigma_i=c_i, \sigma_j=c_j]$ is equal to zero unless $c_i=c_j=c$ for some $c \in [M]$. Thus for the product over $ij \in \{13, 23, 14, 24\}$ in \eqref{eq:prodofEij} to be non-zero all vertices must have the same community assignment to some $c\in[M]$. Hence, 
\begin{eqnarray*} 
\E[Y_{S_2} Y_{T_2}] 
    & = & \sum_{c=1}^M \E[\sY_{12}^2 \: | \: \mono ] \prod_{ij \in \{13, 23, 14, 24\} } \E[\sY_{ij} \: | \: \mono ] \:\prod_{i=1}^4 P(\sigma_i=c).
\end{eqnarray*}
Calculate the conditional expectation of the square.
\begin{eqnarray} 
\label{eq:expsY2} 
\E[ \sY_{ij}^2 \: | \: \sigma_i=c_i, \sigma_j=c_j ]
& \!\!\!= \!\!\!&   \begin{cases}
            q(1-q)+\frac{s}{x_c}(1-2q)  & \mbox{ if $c_i=c_j=c$ for some $c \in [M]$}\\
            q(1-q)                      & \mbox{ otherwise,}
        \end{cases}
\end{eqnarray}
and thus,
\begin{eqnarray*}
\E[\sY_{S_2} \sY_{T_2}] &=&  \sum_{c=1}^M \left[q(1-q) + \frac{s}{x_c}(1-2q) \right]\left(\frac{s}{x_c}\right)^4 \left(\frac{k x_c}{n}\right)^4 \\  &=& M\left(\frac{ks}{n}\right)^4 q(1-q) +   \left(\frac{ks}{n}\right)^4 s (1-2q) \sum_{c=1}^M \frac{1}{x_c}.
\end{eqnarray*}
Since there are at most $n^4$ ways we may pick $S,T\in \Ntri$ with $|S\cap T|=1$, we may conclude that the second term of \eqref{eq:countoverlaps123} is at most
$M k^4s^4 q +   k^4s^5\sum_{c=1}^M \frac{1}{x_c}$.

Lastly we bound the third (and last) term in \eqref{eq:countoverlaps123}, i.e.\ that corresponding to $\dddtriancol$. Similarly to above, fix a set $S$ (and $T$ which entirely overlaps with it), w.l.o.g.\ $S_3= \{12, 13, 23 \}$. Calculate
\begin{eqnarray*} 
\E[\sY_{S_3}^2] 
    & = & P(D_0)\left(q(1-q) \right)^3+ \sum_{i=1}^3 \sum_{c=1}^M P[D_{i,c}] \left(q(1-q)+\frac{s}{x_c}(1-2q)\right)^i \left(q(1-q)\right)^{3-i} \\
    & \leq & P(D_0)q^3 +  \sum_{i=1}^3 \sum_{c=1}^M P[D_{i,c}] \left(q+\frac{s}{x_c}\right)^i q^{3-i} 
\end{eqnarray*}
where $D_{i,c}$ denotes the set of community assignments such that $i$ of $\sY_{12}, \sY_{13}, \sY_{23}$ has distribution Ber$(q+s/x_c)$ (while the others have distribution Ber$(q)$), and $D_0$ denotes the set of community assignments where all three have distribution Ber$(q)$. Observe $D_{1,c}$ is the set of assignments such that two vertices have label $c\in [M]$ and the other vertex has label in $\{\star, 1, \ldots, M\}\backslash\{c\}$ and thus $P(D_{1,c})\leq 3(x_c k/n)^2$. Note $D_{2,c}=\varnothing$. Lastly $P(D_{3,c})=\sum_{c\in M} (x_c k/n)^3$ as $D_{3,c}$ is the community assignment where each of the three vertices has label $c$. Then $P(D_0) = 1 - \sum_c P(D_{1,c}) - \sum_c P(D_{3,c})$. Substituting these bounds for $D_{0}$ and $D_{i,c}$ for $i=1,2,3$ and writing $\rho_c= k x_c/n$ we get
\begin{eqnarray*} 
\E[\sY_{S_3}^2] 
    & \leq & q^3(1-3\rho_c^2-\rho_c^3) + 3\sum_{c=1}^M \rho_c^2 \left(q+\frac{s}{x_c}\right)q^2 +\sum_{c=1}^M \rho_c^3 \left(q+\frac{s}{x_c}\right)^3\\
    & = & q^3 + 3n^{-2} k^2 s q^2  + n^{-3} k^3 \left( 3q^2s\sum_{c=1}^M x_c^2+3qs^2+Ms^3 \right).
\end{eqnarray*}
Since there are $\binom{n}{3}$ ways to pick $S \in \Ntri$ the third term of \eqref{eq:countoverlaps123} is at most $\frac{1}{3}n^3E[\sY_{S_3}^2]$, 
\begin{eqnarray*} 
\tfrac{1}{3}n^3E[\sY_{S_3}^2] 
    & \leq & \tfrac{1}{3}n^3q^3 + n k^2 sq^2 + k^3q^2s + k^3qs^2 + \tfrac{1}{3}Mk^3s^3 
\end{eqnarray*}
where we substituted $\sum_{c \in M} x_c^2 , \sum_{c \in M} x_c^3 \leq 1$. To finish, recall we assumed $M \min_c x_c > C$ for some constant $C$, and note this implies $\max_c 1/x_c < M/C$ and thus $\sum_c 1/x_c^2 < M^2/C$. Apply this to the bounds from the first and second terms of \eqref{eq:countoverlaps123} and we are done.
\end{proof}

\needspace{3\baselineskip}
\section{Relation to recovery: proofs and further discussion.}\label{sec:recoversSWrecovery}
{In this section we prove Proposition~\ref{prop.equivs} showing the relation between the recovery problem studied in~\cite{SW-estimation} and testing between two planted distributions. Loosely, Part 1 of the proposition states that for any recovery problem we may construct an equivalent testing problem, and Part 2 states that for testing problems where the likelihood ratio of the signals exists there is a corresponding recovery problem.

In \cite{SW-estimation} the authors show an upper bound on $\Corr(g(X), \QQ)$ in terms of cumulants $\kappa_\alpha$ where~$\QQ$ is additive Gaussian or Binomial.  We show in Remark~\ref{rem:recoverSWrecovery} below that our Propositions~\ref{prop:adv-gauss} and~\ref{prop:adv-binary} which bound $\Adv$ in terms of cumulant-like quantities $r_\alpha$ recover these cumulant upper bounds proven in~\cite{SW-estimation}.}

\begin{proof}[Proof of~Prop~\ref{prop.equivs}, Part~\ref{recovery_has_testing_equiv}]
We construct $\PP_X$ (the planted part of the distribution $\PP$) by size biasing $\QQ_X$ by $g(X)$, then setting $Y|X$ in $\PP$ to be the same as $Y|X$ in $\QQ$. Suppose first that $\QQ$ has density function $q(x, y)$. Let \[ p(x)= g(x) q(x)/\E_\QQ[g(X)]\] and let $p(y|x)=q(y|x)$. Note this does define a density function since \[ \int p(x, y) dx dy = \int p(y|x) p(x) dx dy = \frac{1}{\E_\QQ[g(X)]} \int g(x) q(y|x) q(x)  dx dy = 1. \]
Comparing the definitions of $\Adv$ in \eqref{eq:adv_repeated} and $\NCorr$ in \eqref{eq:NCorr} it suffices to show that $\E_\PP[h(Y)]=\E_\QQ[g(X) h(Y)]/E_\QQ[g(X)]$ for any polynomial $h$. To see this holds note
\[ \E_\PP[h(Y)] = \int h(y) p(y|x)p(x) dx dy = \frac{1}{ \E_\QQ[g(X)]} \int g(x) h(y) q(y|x)q(x) dx dy = \frac{1}{\E_\QQ[g(X)]} \E_\QQ[g(X) \, h(Y)],  \]
as required.

If we suppose that $\QQ$ is a discrete distibution the proof is similar. Set \[ \PP[X=x] = g(x) \QQ[X=x] / \E_\QQ[g(X)]\] and let $\PP[Y = y| X = x] = \QQ[Y=y| X = x]$. Then, similarly one may check $\PP$ is a probability measure and
\begin{eqnarray*}
 \E_\PP[h(Y)] & = & \sum h(y) \PP(Y=y|X=x)\PP(X=x) \\
 & = & \frac{1}{\E_\QQ[g(X)]}  \sum g(x) h(y) \QQ (Y=y|X= x)\QQ (X=x) \\
 & = & \frac{1} {\E_\QQ[g(X)]} \E_\QQ[g(X) \, h(Y)],
\end{eqnarray*}
as required.
\end{proof}

\begin{proof}[Proof of Prop~\ref{prop.equivs}, Part~\ref{testing_with_likelihood_ratio_has_recovery}]
Suppose first that the joint distribution $(X,Y)\sim \QQ$ (respectively $\sim \PP$) has density function $q(x,y)$ (respectively $p(x,y)$). 

Again, comparing the definitions of $\Adv$ in \eqref{eq:adv_repeated} and $\NCorr$ in \eqref{eq:NCorr} it suffices to show that for any polynomial $h$ we have $\E_\PP[h(Y)]=\E_\QQ[  g(X)h(Y)]/\E_\QQ[g(X)]=\E_\QQ[  g(X)h(Y)]$, since $\E_\QQ[g(X)]=1$.   We may calculate
\[ \E_\QQ[g(X) h(Y)] = \int g(x) h(y) q(y|x)q(x) dx'dy =   \int h(y) q(y|x)p(x) dx dy  \]
where the last line followed since $g(X)=\ell(X)$.
Now recall that $p(y|x)=q(y|x)$ by assumption and thus by the above,
\[ \E_\QQ[g(X) h(Y)] =   \int h(y) p(y|x)p(x) dx dy =  \E_\PP[h(Y)], \]
as required.

If we suppose that $\QQ$ is discrete,  then by definition of $g=g(X)$ and since $Y|X$ has the same distribution under $\PP$ and $\QQ$ \begin{eqnarray*}
\PP_\QQ[ g(X) h(Y) ] 
& = &  \sum_{x, y} g(x) h(y) \PP_\QQ[X=x]\PP_\QQ[Y=y|X=x]  \\
& = &  \sum_{x, y} h(y) \PP_\PP[X =x ] \PP_\QQ[Y=y | X=x ]  \; = \;  \E_\PP[h(Y)] \end{eqnarray*} 
as required.
\end{proof}

\begin{remark}\label{rem:recoverSWrecovery}
In Proposition~\ref{prop.equivs} we have seen that the quantities $\Adv_{\leq D}(\mathbb{P}, \mathbb{Q})$ and $\Corr'_{\leq D}(g(X),\mathbb{Q})$ are the same under some circumstances (i.e., for some triples $g(X), \mathbb{P}, \mathbb{Q}$). In this remark we observe that the upper bounds for  $\Adv_{\leq D}(\mathbb{P}, \mathbb{Q})$ in terms of $r_\alpha$ in this paper imply the upper bounds on $\Corr'$ in terms of $k_\alpha$ in~\cite{SW-estimation}. Note the proofs of the respective results follow the same sequence of steps.  
The $r_\alpha$ enjoy some nice properties, see~Section~\ref{sec:r_algebra}, and it is instructive to see that they extend the cumulants $k_\alpha$ of~\cite{SW-estimation}.

To be precise, for an additive Gaussian model, Theorem~2.2 of~\cite{SW-estimation} states that \[\Corr'^2_{\leq D}(g(X), \mathbb{Q}) \leq \E_\QQ[g(X)]^{-2}  \sum_{ |\alpha|\leq D}  \kappa_{\alpha}^2 /\alpha!\] and we may recover this result via Propositions~\ref{prop.equivs} and~\ref{prop:adv-gauss}.  
We outline this argument now. By Proposition~\ref{prop.equivs}, Part 1 and its proof, there exists a joint distribution $(X,Y) \sim \PP$ such that $g(X)=d\PP_X/d\QQ_X$, and 
    \begin{equation}\label{eq:Adv_equals_Corrp}
        \Adv_{\le D}(\PP,\QQ) =  \NCorr_{\leq D}(g(X),\QQ) 
    \end{equation}
        and 
    \begin{equation}\label{eq:Exp_of_gX}
        \PP_\QQ[ g(X) h(Y) ] =\E_\PP[h(Y)].
    \end{equation}
Now note that by Proposition~\ref{prop:adv-gauss}, $\Adv_{\le D}(\PP,\QQ)\leq \sum_{ |\alpha|\leq D } r_{\alpha}^2/\alpha!$ and thus by~\eqref{eq:Adv_equals_Corrp} it suffices to show we have  $r_\alpha(\PP, \QQ)=\E_\QQ[g(X)]^{-1}\kappa_\alpha(g(X), \QQ)=\kappa_\alpha(g(X), \QQ)$ for the given $g(X)$, $\PP$ and $\QQ$. Recall the non-recursive expression for the $r_\alpha$ from~\eqref{eq:r_non_rec}:
\[ r_\alpha(\PP, \QQ) 
= \sum_{\varnothing \subseteq \delta \subseteq \alpha} \E_{P}[X^\delta] \sum_{\tau \in \mathcal{P}(\alpha\backslash \delta)} (-1)^{|\tau|} \;|\tau|! \prod_{\gamma \in \tau}\E_Q[X^{\gamma}]. \]
The $\kappa_\alpha$ in~\cite{SW-estimation} 
are the joint cumulants of the random variables indexed by $\alpha$ and $g(X)$:
\[\kappa_\alpha(g(X), \QQ) = \sum_{\pi \in \mathcal{P}(\alpha \cup \star)}(-1)^{|\pi|-1} \;(|\pi|-1)! \prod_{\gamma \in \pi}\E_Q[X^{\gamma}]\]
where we denote $X^{\star}=g(X)$. Each partition $\pi$ above contains $\star$ in some part, say $\eta$, and let $\eta'=\eta \setminus \star$. Let $\pi'$ be the partition $\pi \backslash \eta$ and note each $\pi \in \mathcal{P}(\alpha \cup \star)$ gives rise to a unique $\pi' \in \mathcal{P}((\alpha \cup \star) \setminus \eta)=\mathcal{P}(\alpha \setminus \eta')$ with $|\pi'|=|\pi|-1$. Thus,
\[\kappa_\alpha = \sum_{\varnothing \subseteq \eta \subseteq \alpha} \mathbb{E}_\QQ[X^{\eta'}g(X)] \sum_{\pi' \in \mathcal{P}(\alpha \setminus \eta')}(-1)^{|\pi'|} \;|\pi'|! \prod_{\gamma \in \pi'}\E_Q[X^{\gamma}].\]
Now we simply note that $\mathbb{E}_\QQ[X\eta' g(X)]=\mathbb{E}_\PP[X\eta']$ by~\eqref{eq:Exp_of_gX}, and thus $r_\alpha=\kappa_\alpha$ as required. 

\end{remark}

\needspace{3\baselineskip}
\section{Reduction from PDS recovery to testing 1 vs 2 communities}\label{sec:reduction_from1vs2}

In this section we prove that recovery of one planted community is harder than testing between one versus two communities. Thus the results in this paper provides additional evidence for computational hardness for the PDS recovery problem.

We recall Definition~\ref{def:model1} of~$\mathbb{P}_{\rm Gaussian}(n,k,\lambda,M, x)$, which plants $M$ communities of total expected size $k$ with signal strength $\lambda$ and expected proportions of nodes within each community $x=(x_1, \ldots, x_M)$. In the special case, as below, when each community has the same expected size, i.e.\ $x_i=1/M$ for each $i$, we drop the last argument and write~$\mathbb{P}_{\rm Gaussian}(n,k,\lambda,M)$. 
Given a hypothesis testing problem $H_0:$ $Y\sim \mathbb{Q}$ against $H_1:$ $Y\sim\mathbb{P}$, we say an algorithm $B := B(Y)$ \emph{strongly distinguishes} between $H_0$ and $H_1$ if $\mathbb{P}_0(B=1)=o(1)$ and $\mathbb{P}_1(B=0)=o(1)$, or in other words, if the probability of both type 1 and type 2 errors approach zero with growing $n$.
\begin{conjecture}[testing one-vs-two hypothesis]\label{conj.test_1vs2_G}
If $k$ and $\lambda$ scale with $n$ such that $\lambda^2(k^2/n \lor 1)=o(1)$, then no sequence of randomized polynomial time algorithms $B_n:= B_n(Y)$ can strongly distinguish between $H_0:$ $Y\sim\mathbb{Q}=\mathbb{P}_{\rm Gaussian}(n,k,\lambda, 1)$ and $H_1:$ $Y\sim\mathbb{P}=\mathbb{P}_{\rm Gaussian}(n,k,\lambda,2)$.
\end{conjecture}

\begin{remark}
    Results in this paper provide evidence for (a slightly weaker version of) Conjecture~\ref{conj.test_1vs2_G}. To be precise, Theorem~\ref{thm:Gaussian} shows that if $\lambda^2(k^2/n \lor 1)=o(\log^{-5} n)$ then no $O(\log n)$-degree test weakly separates one and two communities (and hence no such test strongly separates).
\end{remark}

We will say that an algorithm, $A_n$, achieves \emph{weak recovery} in $\mathbb{Q}=\mathbb{P}_{\rm Gaussian}(n,k,\lambda, 1)$ if there exists $\eps>0$ (not dependent on $n$) such that  $A_n$ returns a set $I$  such that w.h.p.\ $|I \cap S|\geq \eps k$ and $|I|\leq 1.1 k$, where $S$ is the planted structure in $\mathbb{Q}$.

\begin{theorem}\label{thm:reduction_from1vs2} Assume the testing 1-vs-2 hypothesis (Conjecture~\ref{conj.test_1vs2_G}). If $k$ and $\tilde{\lambda}$ scale with $n$ as
$$k \geq \log \log n, \quad \tilde{\lambda}^2(k^2/n \lor 1)=o(1),\quad 
\text{ and }  \quad \tilde{\lambda} \geq 2k^{-1/2}\log n,$$ 
then no sequence of randomized polynomial-time algorithms $A_n$ achieves weak recovery in \newline$\mathbb{P}_{\rm Gaussian}(n,k, \tilde{\lambda}, 1)$.
\end{theorem}
In the theorem statement we have used $\tilde{\lambda}$ since it will be useful to re-parameterize and take $\lambda=2\tilde{\lambda}$ because we will relate the recovery problem for $\tilde{\lambda}
$ to the testing problem for $2{\tilde{\lambda}}
$.

\paragraph{High-level description of the proof.}

We begin with a high level discussion of the proof and include the full details afterwards.
The intuition for the proof is as follows. Assume the contradictory, that is, assume that exists a sequence of algorithms $A_n$ that achieves weak recovery in $\mathbb{Q} = \mathbb{P}_{\rm Gaussian}(n,k,\tilde{\lambda}, 1)$. We will prove that if this is the case, one can construct a sequence of polynomial-time algorithms that can strongly distinguish between one and two communities, leading to a contradiction.

In more detail, consider that we have access to an algorithm $A := A_n$ that outputs a set of indices $I$ that weakly recovers the \emph{single} planted community under $\mathbb{Q}$. Then in the proof, we will show that we can boost this to a set $I'$ achieving exact recovery. The idea is that we can construct an algorithm $B := B_n$ that can run some (polynomial-time) checks on the output set $I'$, given to us by $A$ and boosting, that w.h.p.\ can distinguish between $\mathbb{P}$ or $\mathbb{Q}$. 
In designing the algorithm $B$, we have to consider both cases, $H_0: Y\sim \mathbb{Q}$ and $H_1: Y\sim \mathbb{P}$. Under $H_0$, we know that $A$ weakly recovers the single community; however, under the two-communities case, $H_1$, we have no guarantees on which set $I$  
will be returned by algorithm~$A$ and this 
is the real challenge in constructing the algorithm $B$.

If we let $Y_{I'}$ denote the data reduced to only the set $I'$, which we have after boosting the $I$ returned by algorithm $A$, then it turns out that the key to constructing the distinguishing algorithm~$B$ is to consider both the trace and sum of $Y_{I'}$. The high-level idea is that, under $H_0$, w.h.p.\ we have a good idea of what values the sum and trace should take, as we are guaranteed that w.h.p.~$I'$ is exactly the single planted community. However, under $H_1$, either the sum will be too small or the trace too large, depending on the makeup of block $I'$ relative to the two communities. 

Let us think more about what can happen with the output of $A$ and boosting in the two-community case, $H_1$. If the returned block $I'$ has large overlap with the pair of planted communities, the trace will likely be $>3k\lambda/2$ (which is much larger than the expected trace in the single-community case) and so we construct $B$ to return $1$ if the trace is very large. On the other hand, if the returned block $I$ has small overlap with the pair of planted communities, then the sum will likely be $<3k\lambda /4$ (which is smaller than the expected sum in the single-community case) and therefore we construct $B$ to also return $1$ when the sum is too small. What the previous two sentences say is that constructing $B$ to use a large trace or small sum threshold will correctly guess $\mathbb{P}$ under $H_1$ (when $\mathbb{P}$ is true), irrespective of the contents of set $I'$.

A nice trick that helps control the random error when we calculate the trace and sum of the returned block is \emph{cloning} as in, for example,~\cite{BBH-reduction}. Note that under $\mathbb{P}$ and $\mathbb{Q}$ the matrix $Y$ can be written as $Y=X+Z$, where $Z$ is a Gaussian matrix independent from the signal $X$, which will take different forms depending on whether the data is generated under $H_0$ or $H_1$.
Cloning means that given $Y$ we may generate three matrices $Y^{(1a)}$, $Y^{(1b)}$ and $Y^{(2)}$ with distributions $Y^{(1a)}=X/2 + Z^{(1a)}$, $Y^{(1b)}=X/2 + Z^{(1b)}$  and $Y^{(2)}=X/\sqrt{2}+Z^{(2)}$. Thus, we have constructed three matrices that are sums of a scaled version of the original signal matrix with independent noise matrices $Z^{(1a)}$, $Z^{(1b)}$ and $Z^{(2)}$ that all follow the same distribution as $Z$.  (The signal strengths are now $\lambda'=\lambda/\sqrt{2}$ for $Y^{(2)}$ and $\tilde{\lambda}=\lambda/2$ for $Y^{(1a)}$ and $Y^{(1b)}$.) This allows us to run algorithm $A$ on matrix $Y^{(1a)}$, returning output $I$, and to boost $I$ to $I'$ using $Y^{(1b)}$, resulting in a set $I'$ that is independent of $Z^{(2)}$. Finally, we use $Y^{(2)}$ for the trace and sum tests.

See Figure~\ref{fig:alg_may_return} for an illustration of the types of returned blocks, $I'$, that can be output by algorithm $A$ and boosting in $H_0:$ $Y\sim \mathbb{Q}$, the single-community case, and $H_1:$ $Y\sim \mathbb{P}$, the two-community case. The figure also shows the approximate expected trace and sum of the submatrix restricted to $I'$ in each of the cases.

\begin{figure}[t!]
\begin{algbox}
\textbf{Algorithm} $\textsc{Weak recovery to testing one vs.\ two}$
\vspace{2mm}

\textit{Inputs}: Matrix $Y \in \mathbb{R}^{n \times n}$, algorithm $A=A_n$.
\begin{enumerate}
\item (\textbf{Pre-processing --- cloning step}.) \newline
Generate independent matrices $Z', Z'' \sim \mathcal{N}(0, 1)^{\otimes n \times n}$ with independent Gaussian entries.  
Compute the matrices
$$Y^{(1)} = \tfrac{1}{\sqrt{2}} (Y + Z') \quad \text{and} \quad Y^{(2)} = \tfrac{1}{\sqrt{2}} (Y - Z'),$$
$$Y^{(1a)} = \tfrac{1}{\sqrt{2}} (Y^{(1)} + Z'') \quad \text{and} \quad Y^{(1b)} = \tfrac{1}{\sqrt{2}} (Y^{(1)} - Z'').$$
\item  (\textbf{Weak recovery}.)  \newline Run algorithm $A$ on $Y^{(1a)}$ to get output $I\subset [n]$. Let $e_I$ be the vector with $i$-th entry equal to 1 if~$i\in I$ and 0 otherwise.

\item (\textbf{Boosting step}.) \newline Let $v=Y^{(1b)}e_I$, then define $I'=\{ \ell \; : \; v_\ell >  k \lambda/(\log \log k) \}$ to get $I'\subset [n]$. 

\item (\textbf{Check size, sum and trace of returned block}.)  \newline Consider submatrix $M=Y^{(2)}_{I'}$ and set $\lambda'=\lambda/\sqrt{2}$. 
If 
\begin{equation}
\label{eq_checks} 
\big|\, |I'|  - k\, \big|  \leq \sqrt{k  \log k}, \;\;\; {\rm sum}(M)\geq  \tfrac{3}{4} \lambda' k^2 
\;\;\;\mbox{and}\;\;\; {\rm tr}(M) \leq \tfrac{3}{2}\lambda' k,
\end{equation}
then output 0. Otherwise, output 1.
\end{enumerate}
\vspace{1mm}
\end{algbox}
\vspace{-2mm}
\caption{Reduction from weak recovery of a planted dense submatrix (PDS) to testing between one-block and two-block models. See also Fig.~\ref{fig:alg_may_return}. 
}\label{alg:reduction}
\end{figure}

\paragraph{Algorithm Definition.}

Mathematically, the algorithm $B=B_n$ is defined as follows --- see also Fig.~\ref{alg:reduction}. Recall that we have been given $Y$, an $n\times n$ matrix, and that $A_n$ is an algorithm that achieves $\eps$-weak recovery in $\mathbb{Q}=\mathbb{P}_{\rm Gaussian}(n, k,  \tilde{\lambda}, 1)=\mathbb{P}_{\rm Gaussian}(n, k, \lambda/2, 1)$, the one-community model.

First, generate $Z'$ and $Z''$, both $n\times n$ symmetric matrices, by independently sampling $Z_{ij}'\sim \mathcal{N}(0,1)$ for $i\leq j$ and setting $Z'_{ji}=Z_{ij}'$ (and independently generate $Z''$ in the same way).
Then set
\[
Y^{(1)}=\frac{Y+Z'}{\sqrt{2}}, \quad \;\;\;\; Y^{(2)}=\frac{Y-Z'}{\sqrt{2}}, \qquad \text{ and } \qquad Y^{(1a)}=\frac{Y^{(1)}+Z''}{\sqrt{2}}, \quad \;\;\;\; Y^{(1b)}=\frac{Y^{(1)}-Z''}{\sqrt{2}}.
\]
Next, run $A = A_n$ on $Y^{(1a)}$ to produce output $I$. Denote by $e_I$ the binary vector with support $I$. Now, let $v=Y^{(1b)}e_I$ and define $I'$ by thresholding: set $i\in I'$ if $v_i >\lambda k/(\log \log k)$. 
    Then, with $\lambda'=\lambda/\sqrt{2}$,
    \begin{equation}\label{eq:alg_def}
        B = 
        \begin{cases}
        0 & \mbox{ if } \; \big | |I'|  - k \big| \leq \sqrt{k  \log k}, \;\; {\rm sum}(Y^{(2)}_{I'})\geq { \tfrac{3}{4}} \lambda' k^2,  
        \;\; \mbox{and} \;\; {\rm tr}(Y^{(2)}_{I'}) \leq \tfrac{3}{2}\lambda' k,   \\
        1 & \mbox{ otherwise}.
        \end{cases}
    \end{equation}

\begin{figure}
      \begin{center}
            \includegraphics[scale=0.8]{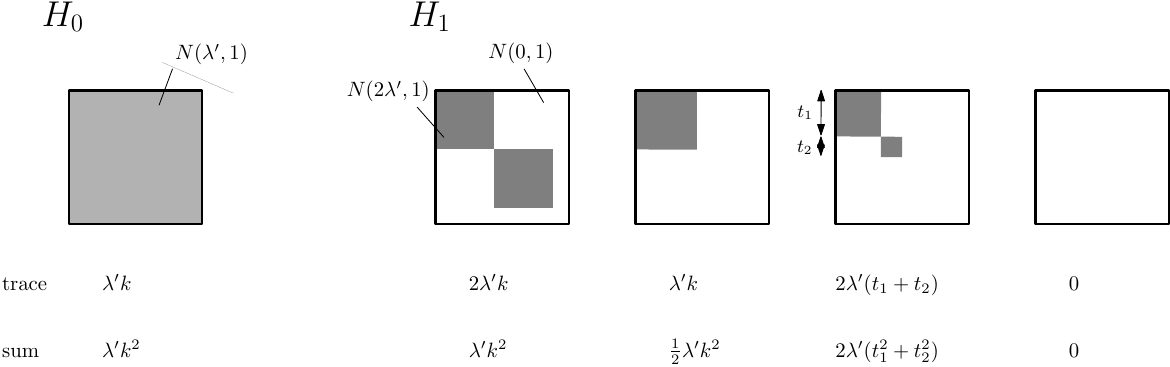}
        \end{center}
    \caption{\small 
    Diagram showing possible returned blocks of the matrix after running the recovery algorithm and boosting. In detail, we depict $Y^{(2)}_{I'}$, i.e.\ the submatrix of $Y^{(2)}$ restricted to the indices $I'$ output by the recovery algorithm $A_n$ run on $Y^{(1a)}$ and boosting on $Y^{(1b)}$ in the proof of Theorem~\ref{thm:reduction_from1vs2}. Here, the shaded regions represent entries~$Y ^{(2)}_{ij}$ where $i,j$ are both in the same planted community. 
    Under $H_0$, we show the recovery algorithm and boosting returns a set $I'$ which is w.h.p.\ exactly the planted structure $S$. Under $H_1$, there are no guarantees on what the algorithm will return, and we present a selection of possibilities: (i)~approximately the entire planted structure (both communities), (ii) one of the two planted communities, (iii) $t_1$ vertices from community one and $t_2$ in community two and (iv) 
    none of the planted structure. For each possibility, we give the expected values (up to leading order) of the trace and sum of the submatrix~$Y^{(2)}_{I'}$. In the proof, we show that we can distinguish between $H_0$ and $H_1$ via the sum and trace of $Y^{(2)}_{I'}$. 
    }
    \label{fig:alg_may_return}
\end{figure}
To show that $B$ distinguishes correctly w.h.p.\ 
requires some auxiliary lemmas which are provided in Section~\ref{sec:clone_tail_bounds}. In particular, we give some known results about cloning from~\cite{BBH-reduction} and 
we also record  
some well-known tail bounds for normal and binomial distributed random variables that we will use in the proof.

\begin{proof}[Proof of Theorem~\ref{thm:reduction_from1vs2}.]
    Let $k$ and $\tilde{\lambda}$ scale with $n$ as prescribed and set $\lambda = 2\tilde{\lambda}$. Also, let $\mathbb{Q}=\mathbb{P}_{\rm Gaussian}(n,k,\lambda, 1)$ and let $\mathbb{P}=\mathbb{P}_{\rm Gaussian}(n,k,\lambda,2)$. Assume for the sake of contradiction that a randomized polynomial-time algorithm $A_n$ achieves weak recovery in $\mathbb{P}_{\rm Gaussian}(n,k,\lambda/2, 1)=\mathbb{P}_{\rm Gaussian}(n,k,\tilde{\lambda}, 1)$. We will show that under this assumption, the algorithm~$B_n$, defined in~\eqref{eq:alg_def}, achieves strong detection between $H_0:$ $Y\sim\mathbb{Q}$ and $H_1:$ $Y\sim\mathbb{P}$, leading to a contradiction.
    
    Recall that to show that $B$ achieves strong detection we must prove that $\mathbb{P}_0(B=1)=o(1)$ and $\mathbb{P}_1(B=0)=o(1)$. In what follows, we call these proofs \textbf{Part 1} and \textbf{Part 2}. 

    \bigskip

    \noindent\textbf{Part 1: proving $\mathbb{P}_0(B=1)=o(1)$.}
    Throughout \textbf{Part 1} we assume the data is generated from $Y=X+Z \sim \mathbb{Q}$, the one-community model. 

    By Lemma~\ref{lem:Gcloning}, $Y^{(1a)}, Y^{(1b)}$ and $Y^{(2)}$ follow the one community model with signal strengths $\tilde{\lambda}=\lambda/2$ for $Y^{(1a)}$ and $Y^{(1b)}$, and $\lambda'=\lambda/\sqrt{2}$ for $Y^{(2)}$, with identical signal locations and independent random noise. Therefore it is equivalent to sample them as follows. 
    
    Construct $X$ in the usual way: i.e.\ independently for each $i\in [n]$, sample the community label $\sigma_i$ such that $\sigma_i=1$ with probability $k/n$ and $\sigma_i= \star$ otherwise. For each $i,j\in [n]$, for $i\leq j$ set $X_{ij}= \lambda \mathbf{1}[\sigma_i=\sigma_j=1]$, and for $i>j$ set $X_{ij}=X_{ji}$. 
    Sample an $n\times n$ matrix $Z^{(1a)}$ in the following way. Independently for each $i,j \in [n]$,  for $i\leq j$ sample $Z^{(1a)}_{ij}\sim \mathcal{N}(0,1)$, and for $i>j$ set $Z^{(1)}_{ij}=Z^{(1)}_{ji}$. Independently sample $Z^{(1b)}$ and $Z^{(2)}$ in the same way.  
    Then set 
    \begin{equation}
    \label{eq:new_Ys}
    Y^{(1a)}=\tfrac{1}{2}X + Z^{(1a)}, \;\;\; \quad Y^{(1b)}=\tfrac{1}{2}X + Z^{(1b)} \;\; \quad \mbox{ and } \quad \;\; Y^{(2)}=\tfrac{1}{\sqrt{2}}X + Z^{(2)}.
    \end{equation}

    Let $I$ be the returned set after running the algorithm $A_n$ on $Y^{(1a)}$, and let $I'$ be the returned set after the boosting step. The important observation is that $I'$ is independent of $Z^{(2)}$. 

    We define a number of events that hold w.h.p.\ and show that if these events occur, then deterministically the algorithm $B$ returns $0$.\\

    \noindent \textbf{Definition of events.}
    Define $E_1$ to be the event that $I'=S_1$, i.e.\ that after the boosting step, the set $I'$ is exactly the planted community, which we have denoted $S_1$. Let $E_2$ be the event that~$ \left \lvert |S_1| - k \right \lvert \leq \sqrt{ k \log k}$.
    The next two events concern the sum and trace of the noise matrix $Z^{(2)}$ induced on the set $I'$. Recall that since $I'$ depends only on $Y^{(1a)}$ and $Y^{(1b)}$, it is independent of $Z^{(2)}$. Let $E_3$ be the event that ${\rm sum}(Z_{I'}^{(2)}) \geq  - 3k \sqrt{\log k}$ 
    and let $E_4$ be the event that ${\rm tr}(Z_{I'}^{(2)}) \leq { 2 \sqrt{ k \log k}}  
    .$ \\

    \noindent \textbf{Claim 1a.}  Under $H_0$, if $E_1, E_2, E_3, E_4$ hold, then the algorithm $B$ outputs 0 for sufficiently large $n$.\\

    \noindent \textbf{Claim 1b}. Under $H_0$, w.h.p.\ the events $E_1$, $E_2$, $E_3$, $E_4$ all hold. \\

    Note that the proof of \textbf{Part 1} follows by these claims which we now prove.

    \bigskip
    \noindent \textbf{Proof of Claim 1a.} To output $B=0$, we need to pass the checks on sum and trace of the block and size of the index set in~\eqref{eq_checks} (see Step 4 of~Fig.~\ref{alg:reduction}): i.e.\ we need that \textbf{(a)} the sum is large, \textbf{(b)} the trace is small and \textbf{(c)} $|I'|\in [k-\sqrt{ k \log k},k+\sqrt{ k \log k}]$. 
    
    Recall that when $E_1$ holds, we have $I'=S_1$. 
    Then since $E_2$ also holds, we have
    \begin{equation}
    \label{eq_I'bound}
    \big||I'| - k \big| =  \big||S_1| - k \big| \leq \sqrt{ k \log k}
    .
     \end{equation}
Notice that the above implies that test \textbf{(c)} is satisfied under $E_1$ and $E_2$.

We will now calculate bounds on the sum and the trace to show that tests \textbf{(a)} and \textbf{(b)} are also passed.
First, from \eqref{eq:new_Ys}, we have ${\rm sum}(Y_{I'}^{(2)}) = \tfrac{1}{\sqrt{2}} {\rm sum}(X_{I'}) + {\rm sum}(Z^{(2)}_{I'})$. Notice that $\tfrac{1}{\sqrt{2}} {\rm sum}(X_{I'}) =\tfrac{1}{\sqrt{2}} \lambda |I'|^2 = \lambda' |I'|^2$, where we recall $\lambda'=\lambda/\sqrt{2}$.  
    Since $E_1$ holds we have $I'=S_1$ and hence from~\eqref{eq_I'bound}, 
    \begin{equation}
    \label{eq:Iprime_bound}
    k - { \sqrt{k\log k}} \leq |I'|   \leq k + { \sqrt{k\log k} } \quad \implies \quad  |I'|^2 = k^2 + {O(k^{3/2}\sqrt{\log k})}.
    \end{equation}
    Therefore, $\tfrac{1}{\sqrt{2}}{\rm sum}(X_{I'}) = \lambda' k^2 + O(\lambda 'k^{3/2}\sqrt{\log k})=\lambda' k^2(1+o(1)).$
    Now, since $E_3$ holds,
    \[ {\rm sum}(Y_{I'}^{(2)}) = \tfrac{1}{\sqrt{2}} {\rm sum}(X_{I'}) + {\rm sum}(Z^{(2)}_{I'})\geq  \lambda' k^2(1+o(1))  - 3k { \sqrt{\log k}} > \tfrac{3}{4}\lambda' k^2,
    \]
    where the last inequality follows since $\sqrt{2}\,\lambda' = \lambda \geq k^{-1/2}\log n$ by assumption; thus we have $3k\sqrt{\log k}=o(\lambda' k^2)$.
    This shows that test \textbf{(a)} is passed, since the requirement in \eqref{eq_checks} is met.

    Similarly, from \eqref{eq:new_Ys}, we have ${\rm tr}(Y_{I'}^{(2)}) = \tfrac{1}{\sqrt{2}} {\rm tr}(X_{I'}) + {\rm tr}(Z^{(2)}_{I'})$.  Since $\lambda'=\lambda/\sqrt{2}$, we find $\tfrac{1}{\sqrt{2}} {\rm tr}(X_{I'}) = \tfrac{1}{\sqrt{2}}  \lambda |I'| =  \lambda' |I'| \leq  \lambda' k(1 + o(1))$ where the inequality follows from \eqref{eq:Iprime_bound}.  Then, since~$E_4$ holds, we have 
    \[ {\rm tr}(Y_{I'}^{(2)}) = \tfrac{1}{\sqrt{2}} {\rm tr}(X_{I'}) + {\rm tr}(Z^{(2)}_{I'}) \leq  \lambda' k(1+o(1))  + 2\sqrt{k  \log k}.\]
    Next, since $ \sqrt{k \log k} \leq {\sqrt{2}}\lambda' k (\sqrt{\log k}/\log n)$ by the assumption that ${ \sqrt{2}}\lambda'\geq k^{-1/2}\log n$, from the above we have
     \[ {\rm tr}(Y_{I'}^{(2)}) \leq  \lambda' k(1+o(1))  + 2\sqrt{k \log k} \leq \lambda' k\left(1+o(1) + \frac{2\sqrt{2\log k}}{\log n}\right)   \leq \tfrac{3}{2} \lambda' k,\]
    where the final inequality holds when $n$ and $k$ are large enough. This shows that test \textbf{(b)} is passed, since the requirement in \eqref{eq_checks} is met.
    Hence, conditional on $E_1, E_2, E_3$ and $E_4$, the interval~$I'$ passes the three checks in~\eqref{eq_checks} and $B$ returns $0$ as required. Therefore, we have proven Claim~1a.\\

    \needspace{4\baselineskip}
    \noindent \textbf{Proof of Claim 1b}.\\

    \noindent\emph{Event $E_1$ holds w.h.p.}
    To prove this, first define $E_0$ to be the event that for some $\eps$ independent of~$n$, both $|S_1 \cap I|\geq \eps k$ and $|I|\leq 1.1k$. Note that by the assumption that $A_n$ achieves weak recovery, $\mathbb{P}_0(E_0)=1-o(1)$. 
    
    Let $v=Y^{(1b)}e_I$. Fix $\ell \in [n]$ and calculate the $\ell$-th entry of vector~$v$ as follows using \eqref{eq:new_Ys} and the definition of $e_I$:
    \begin{eqnarray*}
    v_\ell 
        =  \left[\tfrac{1}{2}Xe_I + Z^{(1b)}e_I \right]_\ell 
         =  \tfrac{1}{2}\sum_{i \in I }X_{\ell i} + \sum_{i \in I }Z^{(1b)}_{\ell i} 
        & = & \begin{cases}
            \tfrac{1}{2}\lambda |I \cap S_1| +  \sum_{i\in I } Z^{(1b)}_{\ell i} & \mbox{if } \ell \in S_1,\\
             \sum_{i \in I } Z^{(1b)}_{\ell i} & \mbox{if } \ell \not\in S_1. \\
        \end{cases}
    \end{eqnarray*}
    Recall that by the design of algorithm $B$, the set $I$ is independent of the noise~$Z^{(1b)}$. Thus, $\sum_{i \in I } Z^{(1b)}_{\ell i}$ has distribution $\mathcal{N}(0, |I|)$ conditional on $I$. 
    Let $\mathcal{E}_\ell$ be the event $|\sum_{i \in I } Z^{(1b)}_{\ell i}|\leq 2\sqrt{k \log n}$.
    Recall that if $E_0$ holds, then $|I|\leq 1.1 k$. 
    Hence, under $E_0$, we apply Lemma~\ref{lem:tail_of_normal_exp} (with 
    $r=\log n+\log \log n$ and $\sigma=\sqrt{1.1k}$, noticing that $\sigma\sqrt{2r}=\sqrt{2.2k (\log n + \log \log n)}\leq 2\sqrt{k\log n}$) to show that w.h.p.\ $\mathcal{E}_\ell$ holds. Indeed,
    \begin{equation}\label{eq:El_v_likely}
    \mathbb{P}_0\left( \mathcal{E}_\ell^c\right)  = \mathbb{P}_0\left( \left|\sum_{i \in I } Z^{(1b)}_{\ell i} \right|\geq 2\sqrt{k \log n}\right) \leq \mathbb{P}_0\left( \left|\mathcal{N}(0, |I|) \right|\geq \sigma \sqrt{2r}\right) \leq 2e^{-r} = \frac{2}{n\log n}.\end{equation}
    Now, if $E_0$ (implying $|I \cap S_1| \geq \eps k$) and $\mathcal{E}_\ell$ hold and $\ell \in S_1$, then for some $\eps>0$ independent of~$n$,
    \[ v_\ell = \tfrac{1}{2}\lambda |I \cap S_1| +  \sum_{i\in I } Z^{(1b)}_{\ell i} \geq \tfrac{1}{2}\lambda \eps k  - 2\sqrt{k\log n}  \geq \frac{\lambda \eps k}{2}  - \frac{2k\lambda}{\sqrt{\log n}} = k\lambda\left(\frac{ \eps}{2}   - \frac{2}{\sqrt{\log n}}\right),\]
    where the second inequality follows as $\lambda \geq k^{-1/2} \log n$. But note that the RHS above is greater than the threshold $k\lambda /(\log \log k)$ for determining $I'$ when $n$ is large (since $k$ grows with $n$). Hence, if $E_0$ and $\mathcal{E}_\ell$ hold, and $\ell \in S_1$ then $\ell \in I'$. Similarly for $\ell\notin S_1$, if $E_0$ and $\mathcal{E}_\ell$ hold, then
    \[ v_\ell =  \sum_{i \in I } Z^{(1b)}_{\ell i} \leq  2\sqrt{k\log n} \leq  2k\lambda/\sqrt{\log n},\]
    and the RHS above is below the threshold $k\lambda /(\log \log k)$, so $\ell \notin I'$. Hence, if $E_0$ and $\mathcal{E}_\ell$ for all $\ell$ hold, then~$I'=S_1$. By~\eqref{eq:El_v_likely}, we may now take a union bound over $\ell \in [n]$, and thus w.h.p.\ $E_0$ and \emph{all} $\mathcal{E}_\ell$ hold and so w.h.p.\ $I' = S_1$. \\

    \noindent\emph{Event $E_2$ holds w.h.p.}   
    The event $E_2$ is an indicator that the single planted community has approximately expected size, i.e.\ that $ \left \lvert |S_1| - k \right \lvert \leq \sqrt{k\log k}$ where $S_1 \subseteq [n]$ is the set of vertices labeled `1', meaning $S_1 = \sigma^{-1}(1)$ 
    for $\sigma$ sampled as above.  Note that $|S_1|\sim {\rm Bin}(n, k/n)$; thus, by Lemma~\ref{lem:conc_for_bins}, 
    \begin{align*}
    \mathbb{P}_0(E_2) &= \mathbb{P}_0\left(\left \lvert |S_1| - k \right\lvert\leq   \sqrt{k \log k}\right) \geq 1 - 2\exp\left\{-\tfrac{1}{3}\log k\right\} = 1 - o(1),
    \end{align*}
    since $k$ is growing with $n$. Notice that by construction, the community assignment is the same for $Y^{(1a)}$, $Y^{(1b)}$ and $Y^{(2)}$. \\
    
    \noindent\emph{Events $E_3$ and $E_4$ hold w.h.p.}
    We have already established $E_1$ and $E_2$ hold w.h.p., and now we show that on $E_1$ and $E_2$, w.h.p.\  $E_3$ and $E_4$ hold as well.

    Next, recall that $I'$ and $Z^{(2)}$ are independent and let $|I'|=t$. Then ${\rm tr}(Z^{(2)}_{I'}) \sim \mathcal{N}(0, t)$ and because the sum of the terms above the diagonal in $Z^{(2)}_{I'}$ is distributed as $\mathcal{N}(0,\binom{t}{2})$, we find ${\rm sum}(Z_{I'}^{(2)}) \sim \mathcal{N}(0, t + 4\binom{t}{2}) = \mathcal{N}(0, 2t^2-t)$. 
    Hence, we may apply Lemma~\ref{lem:tail_of_normal_exp} (with 
    $r=\log k$ and $\sigma =\sqrt{2}t$), to see
    \[ \mathbb{P}_0\left({\rm sum}(Z^{(2)}_{I'}) \geq - 2t\sqrt{\log k} \right) = 1 - \mathbb{P}_0\left(\mathcal{N}(0, 2t^2-t) \geq  2t\sqrt{\log k} \right) \geq 1 - e^{-\log k} = 1 - \tfrac{1}{k}. \]
    Since $E_1$ and $E_2$ hold, 
    we have $t = |I'|=k(1+o(1))$ (see the argument around \eqref{eq_I'bound}). Then w.h.p.\ $E_3$ holds as
    \[ - 2t\sqrt{\log k} =  - 2k(1+o(1))\sqrt{\log k} \geq - 3k \sqrt{ \log k}. \]
    Similarly, since ${\rm tr}(Z^{(2)}_{I'}) \sim \mathcal{N}(0, t)$, by Lemma~\ref{lem:tail_of_normal_exp} (with $r=\log k$ and $\sigma=\sqrt{t}$),
    \[\mathbb{P}_0\left( {\rm tr} (Z_{I'}^{(2)}) \leq \sqrt{2t \log k} \right) = 1 - \mathbb{P}_0\left( \mathcal{N}(0, t) \geq \sqrt{2t \log k} \right) \geq 1- e^{-\log k} = 1 - \tfrac{1}{k}. \]
    Then, as earlier, with $t = |I'|=k(1+o(1))$, we have $\sqrt{2t \log k} \leq 2 \sqrt{k \log k}$; hence, w.h.p.\ $E_4$ holds. 

    This concludes the proof of Claim~1b, and thus we have established Part 1 of the proof.
    \vspace{4mm}

    \noindent\textbf{Part 2:  proving $\mathbb{P}_1(B=0)=o(1)$.}
    Throughout \textbf{Part 2} we assume the data is generated from $Y=X+Z \sim \mathbb{P}$, the two community model. 

    By Lemma~\ref{lem:Gcloning}, $Y^{(1a)}, Y^{(1b)}$ and $Y^{(2)}$ follow the two community model with signal strengths $\tilde{\lambda}=\lambda/2$ for $Y^{(1a)}$ and $Y^{(1b)}$, and $\lambda'=\lambda/\sqrt{2}$ for $Y^{(2)}$, with identical signal locations and independent random noise.
    Equivalently, we may sample $Y^{(1a)}$, $Y^{(1b)}$ and $Y^{(2)}$ as follows. 
    
    Independently for each $i\in [n]$, sample the community label $\sigma_i$ such that $\sigma_i=1$ with probability $k/(2n)$, $\sigma_i=2$ with probability~$k/(2n)$ and $\sigma_i = \star$ otherwise. For each $i,j\in [n]$ with $i\leq j$, set $X_{ij}=2\lambda\mathbf{1}[\sigma_i=\sigma_j=\ell]$ for $\ell\in\{1,2\}$, and for $i>j$ set $X_{ij}=X_{ji}$. 
    Sample an $n\times n$ matrix $Z^{(1a)}$ as in Part 1: independently for each $i,j \in [n]$ with $i\leq j$, sample $Z^{(1a)}_{ij}\sim \mathcal{N}(0,1)$ 
    and for $i>j$, set $Z^{(1)}_{ij}=Z^{(1)}_{ji}$. Independently sample $Z^{(1b)}$ and $Z^{(2)}$ in the same way.  
    Then set \begin{equation}
    \label{eq:clones}
    Y^{(1a)}=\tfrac{1}{2}X + Z^{(1a)}, \;\;\; Y^{(1b)}=\tfrac{1}{2}X + Z^{(1b)} \;\;\mbox{ and }\;\; Y^{(2)}=\tfrac{1}{\sqrt{2}}X + Z^{(2)}.
    \end{equation}
    
    The proof in \textbf{Part 2} is a bit more subtle than the proof in \textbf{Part 1} since we have no guarantees on which set $I$ the algorithm~$A$ will return when given $Y^{(1)}$ with \emph{two} planted communities --- and therefore no knowledge of what set $I'$ will be returned after boosting; see~Figure~\ref{fig:alg_may_return}.

    Notice that if $||I'|-k|>\sqrt{k \log k}$, then the algorithm $B$ returns 1 and we would be done; hence, we assume that $||I'|-k|\leq \sqrt{k \log k}$ throughout.\\ 

    \noindent \textbf{Definition of events.}
    We let $E_1$ be the event that the planted communities have roughly their expected size.    Namely, for $i=1,2$, let $S_i \subseteq [n]$ be the set of vertices labeled `$i$', i.e.\ $S_i = \sigma^{-1}(i)$. Then we let $E_1$ be the event that $|S_1|\leq k/2 + \sqrt{k  \log k} \text{ and } |S_2| \leq k/2 + \sqrt{k  \log k}$. 
    The events~$E_2$, $E_3$ concern the noise in $Z^{(2)}$ for the returned block~$I'$. 
    Define $E_2$ as the event that the sum of the noise in the block is small: ${\rm sum}(Z^{(2)}_{I'}) < { 3k \sqrt{\log k}},$ and  define $E_3$ as the event that the trace of the noise in the block is not too small: ${\rm tr}(Z^{(2)}_{I'}) > - { 2\sqrt{k \log k}}.$

    \bigskip

    \noindent \textbf{Claim 2a.} \emph{Under $H_1$, if $E_1, E_2$ and $E_3$ hold, then the algorithm $B$ outputs 1 for sufficiently large $n$.}\\

    \noindent \textbf{Claim 2b.} \emph{Under $H_1$, w.h.p.\ the events $E_1, E_2, E_3$ all hold for sufficiently large $n$.}\\

    \noindent \textbf{Proof of Claim 2a.} We give a proof by contradiction. Assume that $E_1, E_2$ and $E_3$ all hold and that the algorithm $B$ outputs $B=0$.

    We will see that since $E_2$ and $E_3$ hold, we can find bounds on the sum and trace of~$Y_{I'}^{(2)}$ in terms of the number of vertices of each planted community in~$I'$. In particular, let $t_1 =t_1(I') = |S_1 \cap I'|$ and let $t_2=t_2(I')=|S_2 \cap I'|$. See also Figure~\ref{fig:alg_may_return}. First, considering the diagonal elements of $X_{I'}$, we have that an element equals zero if it is in $I'$ but not in $S_1$ or $S_2$ and equals $2\lambda$ if it is in $S_1$ or $S_2$.  Hence, ${\rm tr}(X_{I'}) = 2 \lambda(t_1 + t_2)$. For off-diagonal elements $[X_{I'}]_{i,j}$, these will equal $2\lambda$ only if both $i$ and $j$ belong to the same community. Hence, ${\rm sum}(X_{I'}) = 2 \lambda(t_1 + t_2) + 2 \times 2 \lambda(\binom{t_1}{2} + \binom{t_2}{2}) = 2 \lambda(t_1^2 + t_2^2)$.

    Now, notice from \eqref{eq:clones}, the fact that $\lambda' = \lambda/\sqrt{2}$, the argument above for ${\rm sum}(X_{I'})$ and $E_2$ that
    \begin{equation}
    \label{eq:sum_upper_bound}
    {\rm sum}(Y^{(2)}_{I'}) = \frac{1}{\sqrt{2}}{\rm sum}(X_{I'}) + {\rm sum}(Z^{(2)}_{I'}) \leq  2\lambda' (t_1^2 + t_2^2)  + { 3k\sqrt{\log k}}.
    \end{equation}
    Similarly, notice from \eqref{eq:clones}, the argument above for ${\rm tr}(X_{I'})$ and $E_3$ that
    \begin{equation}\label{eq:trace_lower_bound}
    {\rm tr}(Y^{(2)}_{I'}) =  \frac{1}{\sqrt{2}}{\rm tr}(X_{I'}) + {\rm tr}(Z^{(2)}_{I'}) = 2\lambda'(t_1+t_2) + {\rm tr}(Z^{(2)}_{I'})  \geq 2\lambda' (t_1 + t_2) - {2\sqrt{k\log k}}.
    \end{equation}
    Since $\sqrt{2}\lambda' \geq k^{-1/2} \log n$ by assumption we have,
    \begin{equation}\label{eq:trace_lower_bound_2}
    {\rm tr}(Y^{(2)}_{I'}) \geq 2\lambda' (t_1 + t_2) - { 2\sqrt{k\log k}} \geq 2\lambda' (t_1 + t_2) -  2\sqrt{2}\,\lambda' k \frac{\sqrt{\log k}}{ \log n} { \; = 2\lambda'(t_1+t_2) + o(\lambda' k)},
    \end{equation}
    where the last equality follows since $\sqrt{\log k}=o(\log n)$.

    In the next part of the proof, we show that since we assumed $B$ outputs 0, the trace of~$Y_{I'}^{(2)}$ must be small; thus, $t_1+t_2$ is also small from \eqref{eq:trace_lower_bound_2}. Then we show that  $t_1+t_2$ being small implies that $Y^{(2)}_{I'}$ will have small sum using \eqref{eq:sum_upper_bound}. This will be the contradiction -- we necessarily fail the sum test and the algorithm outputs 1.

    Now, since we have assumed that $B$ output 0, using the checks  in~\eqref{eq_checks}, the following must hold:
    \begin{equation}\label{eq:fool_B_tr_test}
        {\rm tr}(Y_{I'})\leq \tfrac{3}{2}\lambda' k,
    \end{equation}
    Then by \eqref{eq:trace_lower_bound_2} and~\eqref{eq:fool_B_tr_test}, we may deduce that
    \begin{equation}\label{eq:gotta_be_close_to_planted_v1}
    \tfrac{3}{2}\lambda' k \geq 2\lambda' (t_1 + t_2) { +o(\lambda' k)} \quad \implies \quad   t_1 + t_2  \leq  \tfrac{3}{4}k + { o(k)}.
    \end{equation}
    Next, note that since $E_1$ holds and each $t_i$ is  the cardinality of $S_i \cap I' \subseteq S_i$, then we know that both $t_1$ and $t_2$ are upper bounded by $k/2 + \sqrt{k  \log k}$. Also, recall that for any $y_1, y_2$ with $0 \leq  y_1, y_2 \leq \eta$ we have $y_1^2+y^2_2 \leq \eta^2 + (y_1+y_2-\eta)^2$.
    Hence, for large enough $n$ (since $k$ grows with $n$) 
    \begin{align*}
    t_1^2+t_2^2 &\leq (\tfrac{1}{2}k+\sqrt{k  \log k})^2+ (t_1 + t_2 - \tfrac{1}{2}k-\sqrt{k \log k})^2  \\
    &\overset{(a)}{\leq} (\tfrac{1}{2}k+\sqrt{k  \log k})^2+ (\tfrac{3}{4}k + { o(k)}  - \tfrac{1}{2}k-\sqrt{k \log k})^2 \\
    &\leq (\tfrac{1}{2}k+{ o(k)})^2+ (\tfrac{1}{4}k + { o(k)} )^2  \\
    & < \tfrac{1}{3}k^2
    \end{align*}
    where step $(a)$ follows from \eqref{eq:gotta_be_close_to_planted_v1}. 
    Thus by~\eqref{eq:sum_upper_bound}, for large enough $n$,
    \begin{equation}\label{eq:nearly_contradiction}
    \notag {\rm sum}(Y^{(2)}_{I'}) < \tfrac{2}{3}\lambda' k^2  + {  3k\sqrt{\log k}}.
    \end{equation}
    Recalling that $\sqrt{2}\lambda'\geq k^{-1/2}\log n$,
    \begin{equation}\label{eq:contradiction}
    \notag {\rm sum}(Y^{(2)}_{I'}) < \tfrac{2}{3}\lambda' k^2  +  3\lambda' k^{3/2} \sqrt{2\log k} \log^{-1} n { \; < \tfrac{3}{4}\lambda' k^2},
    \end{equation}
    for large enough $n$. Therefore, inspecting the definition in~\eqref{eq:alg_def}, $I'$ fails the sum check and thus $B$ returns `1', a contradiction, and we have proven Claim~2a.\\

    \noindent\textbf{Proof of Claim 2b.}\\

    \noindent\emph{Event $E_1$ holds w.h.p.} Notice that $|S_1|, |S_2| \sim {\rm Bin}(n, k/(2n))$; thus, by Lemma~\ref{lem:conc_for_bins}, w.h.p.\ $|S_1| \leq k/2 + \sqrt{k  \log k}$ for large~$k$, and similarly for $S_2$. Hence, $\mathbb{P}_1(E_1)=1-o(1)$. \\

    \noindent\emph{Events $E_2$ and $E_3$ hold w.h.p.}
    Let $t=|I'|$ then, as in \textbf{Part 1} of this proof, since $Z^{(2)}$ is independent of $I'$, ${\rm sum}(Z_{I'}^{(2)})\sim \mathcal{N}(0, 2t^2 - t)$ and ${\rm tr}(Z_{I'}^{(2)})\sim \mathcal{N}(0, t)$ where $t=|I'|$. Also, we may assume that $t\in k\pm \sqrt{k   \log k}$. Thus, the events $E_2$ and $E_3$ hold by applying Lemma~\ref{lem:tail_of_normal_exp}, as in \textbf{Part 1} of this proof. \\

    This completes the proof of Claim~2b, which concludes the proof of the theorem.\end{proof}

    \subsection{Lemmas used to prove Theorem~\ref{thm:reduction_from1vs2}}\label{sec:clone_tail_bounds}

    For the Gaussian cloning trick, we note Lemma 10.2 of~\cite{BBH-reduction}, see also Figure 23 of that paper. Following notation of that paper, given a distribution $\PP$, we denote by $\PP^{\otimes n}$ (respectively $\PP^{\otimes n \times n}$) the distribution $(X_1, \ldots, X_n)$ (respectively the distribution on $n\times n$ matrix) where the $X_i$ are i.i.d.\ and $X_i \sim \PP$. We write $\overset{d}{=}$ to denote that the two sides are equal in distribution.
    \begin{lemma}[Gaussian cloning]\label{lem:Gcloning}
    Given a random matrix $M \overset{d}{=} A + \mathcal{N}(0, 1)^{\otimes n \times n}$ for any fixed matrix $A \in \mathbb{R}^{n\times n}$, independently sample $G \sim \mathcal{N}(0, 1)^{\otimes n \times n}$ and set $M_1=(M+G)/\sqrt{2}$ and $M_2=(M-G)/\sqrt{2}$. Then $M_1$ and $M_2$ are independent, both $\overset{d}{=} \frac{1}{\sqrt{2}}A + \mathcal{N}(0, 1)^{\otimes n \times n}$.
    \end{lemma}
    The following standard result will help us prove bounds on the noise in the returned block, i.e.\ for $Z^{(1b)}_I$ and $Z^{(2)}_{I'}$. 
    \begin{lemma}\label{lem:tail_of_normal_exp} Let $X$ be a random variable that is distributed as $\mathcal{N}(0, \tilde{\sigma}^2)$ for some $\tilde{\sigma} \leq \sigma$. Then \[ \mathbb{P}\big(X \geq\sigma \sqrt{2 r} \; \big) \leq e^{-r}.\]
    \end{lemma}
    Finally, we use a tail bound for binomial random variables that follows, for example, from Theorems~2.1 and 2.8 of \cite{JLRbook}. 
    \begin{lemma}\label{lem:conc_for_bins}
    Let random variables $X_1, \ldots, X_m$ be independent, with $0 \leq  X_j \leq 1$ for each $j$. Let $S = \sum_{j=1}^{m} X_j$ and $\mu = \E(S)$. Then, for $0 < x \leq \sqrt{\mu}$, 
    \[ \mathbb{P}(|S-\mu| \geq x \sqrt{\mu}) \leq 2 e^{-\frac13 x^2}.
    \]
    \end{lemma}

    \section*{Acknowledgements} 
    This work began when the authors were visiting the Simons Institute for the Theory of Computing during the program on Computational Complexity of Statistical Inference. We are grateful to Guy Bresler for helpful discussions.

    \typeout{} 

    \bibliographystyle{alpha}
    \bibliography{main}

    \end{document}